\documentclass{amsart}

\usepackage{a4wide,url}
\usepackage{hyperref}
\usepackage{todonotes}
\usepackage{pdfsync} 
\usepackage{longtable}
\usepackage[tableposition=below]{caption}
\usepackage{pdfpages}
\usepackage{tikz,xy}
\xyoption{all}

\usepackage{tabularx,ragged2e,booktabs,caption}
\usepackage{amsthm}                   
\usepackage{amsmath}                  
\usepackage{amsfonts}                 
\usepackage{amssymb}                  
\usepackage{mathrsfs}                 
\usepackage{mathtools}								
\usepackage{pstricks-add}
\usepackage{xy}
\usepackage{graphicx,subfigure}
\usepackage{hyperref}
\xyoption{all}
\usepackage{multirow}
\usepackage{tikz}  
\usepackage[figuresright]{rotating}
\usepackage[vlines]{tabularht}

\theoremstyle{definition}
\newtheorem{definition}{Definition}[section]
\newtheorem{remark}[definition]{Remark}

\theoremstyle{plain}
\newtheorem{theo}[definition]{Theorem}
\newtheorem{lem}[definition]{Lemma}

\newtheorem{prop}[definition]{Proposition}
\newtheorem*{theo*}{Theorem}

\numberwithin{equation}{section}

\newcommand{\R}{\ensuremath{\mathbb{R}}}     	
\newcommand{\N}{\ensuremath{\mathbb{N}}}     	
\newcommand{\Z}{\ensuremath{\mathbb{Z}}}     	
\newcommand{\D}{\ensuremath{\mathcal{D}}}     	
\newcommand{\nI}{\ensuremath{\mathcal{I}}}
\newcommand{\nP}{\ensuremath{\mathcal{P}}}
\newcommand{\nC}{\ensuremath{\mathcal{C}}}
\newcommand{\nQ}{\ensuremath{\mathcal{Q}}}
\newcommand{\nS}{\ensuremath{\mathcal{S}}}   	

\newcommand{\Bold}{\boldsymbol}
\newcommand{\B}{\ensuremath{\boldsymbol{B}}}

\newcommand{\level}[1]{\mathrm{level}(#1)}
\newcommand{\sub}{t}

\newcommand*\circled[1]{\tikz[baseline=(char.base)]{
            \node[shape=circle,draw,inner sep=1.2pt] (char) {#1};}}

\title{On self-affine tiles that are homeomorphic to a ball}

\author{J\"org M. Thuswaldner}
\author{Shu-Qin Zhang}

\thanks{Part of this paper was written during the conference ``Numeration 2019'' which took place at the Erwin Schr\"odinger Institute in Vienna. The authors acknowledge the hospitality and the convenient working conditions that were provided there. Both authors are supported by the grant FWF W1230 funded by the Austrian Science Fund.}

\address[J.M.T.]{Chair of Mathematics and Statistics, University of Leoben, Franz-Josef-Strasse 18, A-8700 Leoben, Austria}
\email{joerg.thuswaldner@unileoben.ac.at}
\address[S.-Q.Z.]{School of Mathematics and Statistics, Zhengzhou University, 100 Science Avenue, Zhengzhou, Henan 45001, People’s Republic of China}
\email{sqzhang@zzu.edu.cn}

\date{\today}

\begin{document}

\begin{abstract}
Let $M$ be a $3\times 3$ integer matrix which is expanding in the sense that each of its eigenvalues is greater than $1$ in modulus and let $\mathcal{D} \subset \Z^3$ be a {\it digit set} containing $|\det M|$ elements. Then the unique nonempty compact set $T=T(M,\D)$ defined by the set equation $MT=T+\D$ is called an {\em integral self-affine tile} if its interior is nonempty. If $\D$ is of the form $\D=\{0,v,\ldots, (|\det M|-1)v\}$ we say that $T$ has a {\em collinear digit set}. The present paper is devoted to the topology of integral self-affine tiles with collinear digit sets. In particular, we prove that a large class of these tiles is homeomorphic to a closed $3$-dimensional ball. Moreover, we show that in this case $T$ carries a natural CW complex structure that is defined in terms of the intersections of $T$ with its neighbors in the lattice tiling $\{T+z\,:\, z\in \Z^3\}$ induced by $T$. This CW complex structure is isomorphic to the CW complex defined by the {\em truncated octahedron}. 
\end{abstract}

\maketitle

\section{introduction}
\subsection{Context of the paper}
The present paper is devoted to the study of the topology of $3$-dimensional self-affine tiles. 

Let $M\in \Z^{n\times n}$ be an integer matrix which is expanding in the sense that each of its eigenvalues has modulus strictly greater than one. Moreover, let $\D \subset \Z^n$ be a {\it digit set} with $|\det M|$ elements. 
\begin{figure}[h]
\includegraphics[angle=0,trim={0 0 00 0},clip=true,width=.4\textwidth]{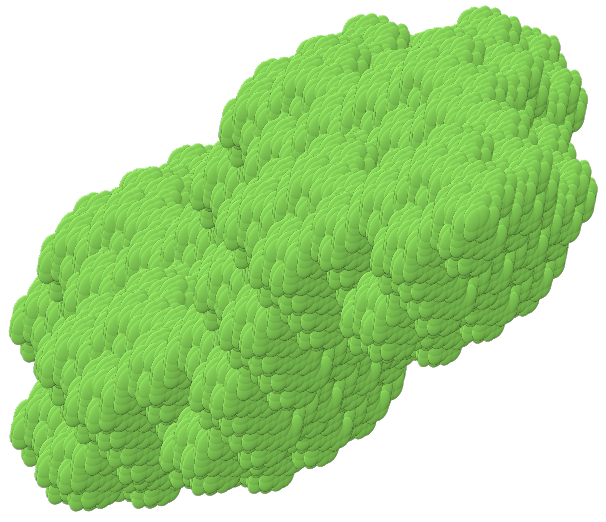} 
\hskip 0.5cm
\includegraphics[trim={0 0 00 0},clip=true,origin=c,width=.5\textwidth]{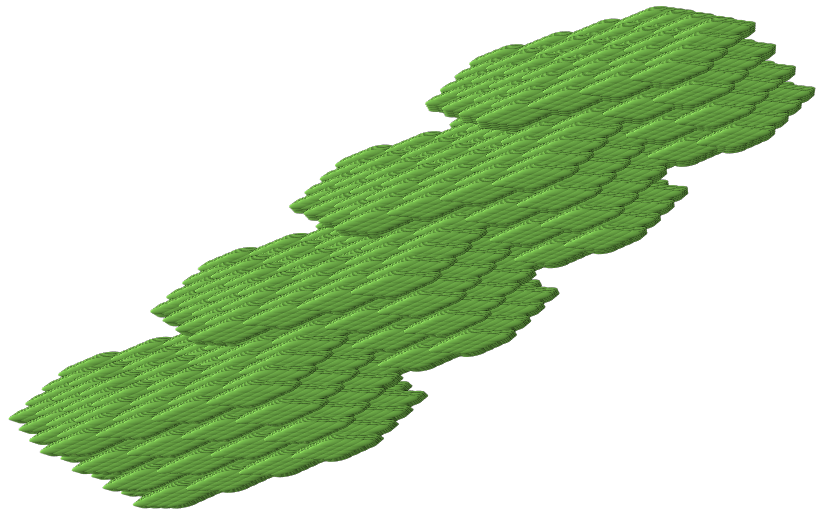}  
\caption{Two examples of $3$-dimensional self-affine tiles.\label{fig:3d}}
\end{figure}
Then it follows from the theory of iterated function systems (see {\it e.g.}~Hutchinson~\cite{Hutchinson81}) that there is a unique nonempty compact set $T=T(M,\D)$ such that 
\begin{equation}\label{eq:setequationPRAE}
MT=T+\D.
\end{equation}
If $T$ has nonempty interior then it is called an {\it integral self-affine tile}, or just a {\it self-affine tile} for short. If $\D$ is a complete set of coset representatives of the residue class ring $\Z^n/M\Z^n$, it is called a {\it standard digit set}. For standard digit sets it is known that the nonempty compact set $T$ defined by \eqref{eq:setequationPRAE} always has nonempty interior (see Bandt~\cite{Bandt91}). 

Self-affine tiles have been studied systematically since the 1990ies when 
Bandt~\cite{Bandt91}, Kenyon~\cite{Kenyon92}, Gr\"ochenig and Haas~\cite{GroechenigHaas94}, as well as Lagarias and Wang~\cite{LagariasWang96b,LagariasWang96a,LagariasWang97} proved fundamental results on these objects. Since that time the research on self-affine tiles developed in many different directions and they play a role in various branches of mathematics like in the theory of dynamical systems, in number theory, and in Fourier analysis and the construction of wavelets. The present paper is concerned with the topology of self-affine tiles. Since the seminal paper of Hata~\cite{Hata85}, the topology of self-affine sets in general, and of self-affine tiles in particular, has been thoroughly studied. Connectivity properties of self-affine tiles can be treated in a satisfactory way in arbitrary dimension $n$ (see for instance Kirat and Lau~\cite{KiratLau00}). Further investigation of their topology often relies on the Jordan curve theorem. For this reason, many papers on the topology of self-affine tiles are restricted to the $2$-dimensional case. We refer for instance to Bandt and Wang~\cite{BandtWang01} or Leung and Lau~\cite{LauLeung07} where homeomorphy to a disk was investigated, or to Ngai and Tang~\cite{NgaiTang04,NgaiTang05} for the study of self-affine tiles with disconnected interior. Another interesting direction of research which has relations to the {\it Fuglede conjecture} ({\it cf.~e.g.}~\cite{Fuglede:74,Tao:04}) is the characterization of all digit sets $\D$ that give rise to a self-affine tile $T(M,\D)$ for a given expanding integer matrix $M$, see for instance An and Lau~\cite{AnLau19}, Lai {\it et al.}~\cite{LaiLauRao:17}, and the survey by Lai and Lau~\cite{LaiLau17}. 

The present paper is devoted to the topology of $3$-dimensional self-affine tiles. The systematic topological study of the $3$-dimensional case was initiated some years ago when Bandt~\cite{Bandt10} considered the combinatorial topology of some $3$-dimensional self-affine tiles. Later, Conner and Thuswaldner~\cite{ConnerThuswaldner0000} gave criteria for a self-affine tile to be a closed $3$-dimensional ball and Deng~{\it et al.}~\cite{DengLiuNgai18} dealt with self-affine tiles of a special form and showed that they are $3$-dimensional balls.  Kamae~{\it et al.}~\cite{KLT:15} investigated a particular class of $n$-dimensional self-affine tiles. Recently, Thuswaldner and Zhang~\cite{TZ:19} studied a natural class of $3$-dimensional self-affine tiles and proved that their boundary is homeomorphic to a $2$-sphere. It is this class of tiles that we are interested in. Indeed, we want to explore if these tiles are indeed homeomorphic to a $3$-dimensional ball, which means that we have to exclude pathologies like the Alexander horned sphere which is known to occur in the context of self-affine tiles (see \cite[Section~8.2]{ConnerThuswaldner0000}).

\subsection{Description of the main results}
Our aim is to prove that a large class of well-known $3$-di\-men\-sio\-nal self-affine tiles is homeomorphic to a closed $3$-dimensional ball. Moreover, we will show that each tile in this class  carries a natural CW complex structure~(see {\it e.g.}\ Hatcher~\cite[p.~5]{Hatcher:02} for the definition of a CW complex). 

Before we state our main results, we introduce some notation. Let $M$ be an expanding $3\times 3$ integer matrix and let $\D \subset \Z^3$ be a {\em digit set} such that the unique nonempty compact set $T=T(M,\D)$ defined by the set equation
\begin{equation}\label{eq:setequation}
T= \bigcup_{d\in \D}M^{-1}(T+d)
\end{equation}
has nonempty interior. Then $T$ is a {\em self-affine tile}.
Define the set of \emph{neighbors} of $T$ by 
\begin{equation}\label{eq:neigh}
\nS=\{\alpha\in \Z[M,\D]\setminus\{0 \}\;: \; T\cap (T+\alpha)\neq \emptyset\}.
\end{equation}
Here
\[
\Z[M,\D] = \Z[\D,M\D, M^{2}\D] \subseteq \Z^3
\]
is the smallest $M$-invariant lattice containing $\D$. This definition is motivated by the fact that the collection $\{T+ z \,:\, z \in \Z[M,\D]\}$ often tiles the space $\R^3$ with overlaps of Lebesgue measure $0$ ({\it cf.\ e.g.}~Lagarias and Wang~\cite{LagariasWang97}). The translated tiles $T+\alpha$ with $\alpha\in\nS$ are then those tiles which ``touch'' ({\it i.e.}, have nonempty intersection with) the ``central tile'' $T$ in this tiling. It is clear that $\nS$ is a finite set since $T$ is compact by definition and $\Z[M,\D]$ is discrete. For the sets in which $T$ intersects with one given other tile we use the notation
\begin{equation}\label{eq:bgamma}
\B_{\alpha}=T\cap(T+\alpha) \qquad(\alpha \in \Z[M,\D]\setminus\{0\}). 
\end{equation}
More generally, for $\ell \ge 0$ we define the set of points in which $T$ intersects $\ell$ given other tiles by
\begin{equation}\label{eq:bbold}
\B_{\Bold{\alpha}} = \B_{\{\alpha_1,\ldots, \alpha_{\ell}\}} = T \cap (T+\alpha_1) \cap \cdots \cap (T+\alpha_{\ell}) \qquad (\Bold{\alpha}=\{\alpha_1,\ldots, \alpha_{\ell}\}\subset \Z[M,\D]\setminus\{0\}).
\end{equation}
Note in particular that $\B_\emptyset=T$.
Compactness of $T$ and discreteness of $\Z[M,\D]$ again ensures that there exist only finitely many $\Bold{\alpha}\subset \Z[M,\D]\setminus\{0\}$ satisfying $\B_{\Bold{\alpha}}\not=\emptyset$.

We will be interested in the following class of self-affine tiles. Let $M$ be an expanding $3\times 3$ integer matrix. We call $\D \subset \Z^3$ a {\em collinear digit set} for $M$ if there is a vector $v\in\Z^3\setminus\{0\}$ such that
\begin{equation}\label{eq:collinearD}
\D=\{0,v,2v,\ldots, (|\det M|-1)v\}.
\end{equation}
If $\D$ has this form, a self-affine tile\footnote{Note that we assume here that $T(M,\D)$ is a self-affine tile. This does not follow from the collinearity of $\D$.}
 $T=T(M,\D)$ is called a {\em self-affine tile with collinear digit set} (such tiles have been studied intensively in recent years, {\it cf.~e.g.}~\cite{LauLeung07,TZ:19}). 

For $k\ge 0$ denote the $k$-dimensional unit ball by $\mathbb{D}^k=\{x\in \mathbb{R}^k: \Vert x \Vert_2 \le 1\} \subset \mathbb{R}^k$  ($\Vert\cdot\Vert_2$ is the Euclidean norm). We note that $\mathbb{D}^0$ is a single point.  A {\em closed $k$-cell} or {\em $k$-ball} is a topological space that is homeomorphic to $\mathbb{D}^k$.

Our first main result shows that a large class of self-affine tiles with collinear digit sets are $3$-balls.

\begin{theo}\label{thm:ball}
Let $T=T(M,\D)$ be a $3$-dimensional self-affine tile with collinear digit set and assume that the characteristic polynomial $\chi(x)=x^3+Ax^2+Bx+C$ of $M$ satisfies $1 = A\le B < C$.
If $T$ has $14$ neighbors then  $T$ is a $3$-ball.
\end{theo}

\begin{remark}\label{rem:1}
Let $T=T(M,D)$ be a $3$-dimensional self-affine tile with collinear digit set. If the coefficients $A,B,C$ of the characteristic polynomial $\chi(x)=x^3+Ax^2+Bx+C$ of $M$ satisfy $1 = A\le B < C$ then the matrix $M$ is expanding (see~\cite[Lemma~2.2]{TZ:19}). Moreover, according to~\cite[Theorem~1.1]{TZ:19} the collection $\{T+ \alpha \,:\, \alpha \in \Z[M,\D]\}$ tiles the space $\R^3$ with overlaps of Lebesgue measure~$0$. 
\end{remark}

\begin{remark}\label{rem:2}
According to \cite[Theorem~1.4]{TZ:19} a $3$-dimensional self-affine tile $T=T(M,D)$ with collinear digit set and characteristic polynomial $\chi(x)=x^3+Ax^2+Bx+C$ of $M$ satisfying $1 \le A\le B < C$ has $14$ neighbors if and only if one of the following conditions holds:
\begin{itemize}
\item $1\leq A<B<C$ and $B\geq 2A-1, C\geq 2(B-A)+2$;
\item $1\leq A<B<C$ and $B<2A-1, C\geq A+B-2$. 
\end{itemize}
We believe that similar criteria can also be established if negative coefficients are allowed.
\end{remark}

\begin{remark}\label{rem:3}
We conjecture that, apart from sporadic cases (as, for instance, the ones studied in~\cite{Bandt10}), $3$-dimensional self-affine tiles with collinear digit set having more than $14$ neighbors are not homeomorphic to a $3$-ball. In the $2$-dimensional case, only self-affine tiles with a small number of neighbors have a nice topological structure (see~\cite{BandtWang01};  we refer to~\cite{LuoThuswaldner06,NgaiTang04,NgaiTang05} for $2$-dimensional tiles with wild topology). 
\end{remark}

Our second main result shows that the sets $\B_{\Bold{\alpha}}$ defined in \eqref{eq:bbold} provide a natural CW complex structure on $T$. 

\begin{theo}\label{thm:CW}
Let $T=T(M,\D)$ be a $3$-dimensional self-affine tile with collinear digit set and assume that the characteristic polynomial $\chi(x)=x^3+Ax^2+Bx+C$ of $M$ satisfies $1 = A\le B < C$. 
If $T$ has $14$ neighbors then $T$ carries the following natural CW complex structure. 
\begin{itemize}
\item The closed $0$-cells are the $24$ nonempty sets $\B_{\{\alpha_1,\alpha_2,\alpha_3\}}$ with $\{\alpha_1,\alpha_2,\alpha_3\}\subset \Z[M,\D]\setminus\{0\}$. 
\item The closed $1$-cells are the $36$ nonempty sets $\B_{\{\alpha_1,\alpha_2\}}$ with $\{\alpha_1,\alpha_2\}\subset \Z[M,\D]\setminus\{0\}$.
\item The closed $2$-cells are the $14$ nonempty sets $\B_{\alpha_1}$ with $\alpha_1\in \nS$. 
\item The closed $3$-cell is $\B_\emptyset$.
\end{itemize}
For $i\in\{1,2,3\}$, the closed $i$-cell $\B_{\boldsymbol{\alpha}}$, $\# \boldsymbol{\alpha} = 3 -i$, is attached to the $(i-1)$-skeleton $T^{i-1}$ by attaching its boundary $\partial \B_{\boldsymbol{\alpha}}$ (as a manifold) to the $(i-1)$-sphere
\[
\bigcup_{\alpha \not \in \boldsymbol{\alpha}} \B_{\boldsymbol{\alpha} \cup \{\alpha\}}.
\] 
This CW complex is isomorphic to the natural CW complex structure of a truncated octahedron.
\end{theo}

\begin{remark}
In the literature (see {\it e.g.} Hatcher~\cite[p.~5]{Hatcher:02}), an (open) $k$-cell of a CW complex is a topological space that is homeomorphic to the open unit ball in $\mathbb{R}^k$ for $k\ge 0$ (a $0$-cell is a single point). The $k$-cells of the CW complex defined in Theorem~\ref{thm:CW} are the nonempty sets $\mathrm{Int}(\B_{\boldsymbol{\alpha}})$ with $|\alpha|=3-k$ ($0\le k \le 3$). Here, for a $k$-manifold $\mathcal{M}$ with boundary, $\mathrm{Int}(\mathcal{M})$ denotes the set of $x\in \mathcal{M}$ having a neighborhood that is homeomorphic to a $k$-cell (contrary to the topological interior $X^\circ$ of a set $X$ w.r.t. some ambient space). We use closed cells for notational convenience.
\end{remark}

\begin{figure}[h]
\includegraphics[angle=0,trim={0 80 0 40},clip=true,width=0.7\textwidth]{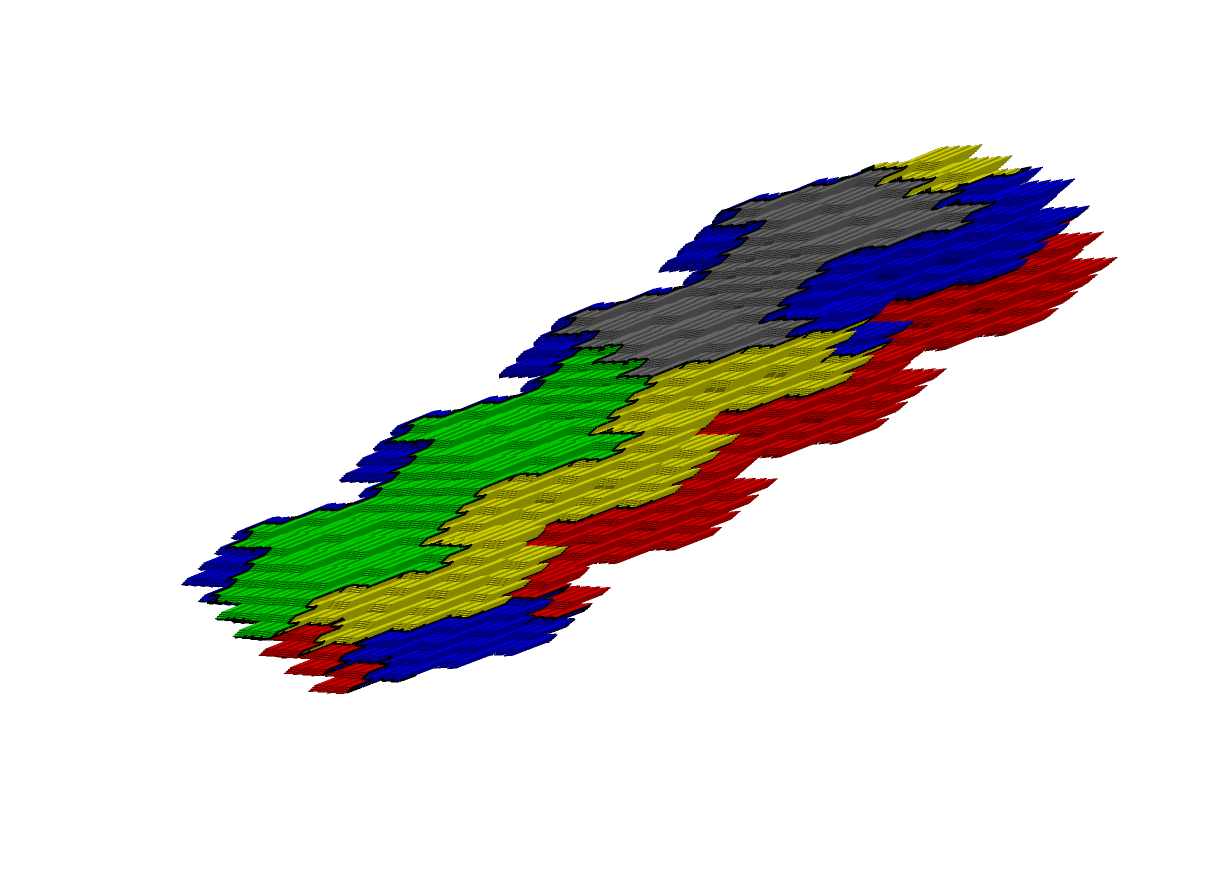} 
\caption{The CW complex structure of a self-affine tile $T$.\label{fig:CWtile}}
\end{figure}

In Figure~\ref{fig:CWtile} we visualize the CW complex structure of the self-affine tile $T$ on the right hand side of Figure~\ref{fig:3d}. The whole tile $T=\B_\emptyset$ is a closed $3$-cell. Each of the patches is homeomorphic to a closed $2$-cell $\B_{\alpha}$ for some $\alpha \in \nS$. The union of these patches forms the $2$-sphere $\partial T$. Two distinct closed $2$-cells meet in a closed $1$-cell $\B_{\{\alpha_1,\alpha_2\}}$, and three closed $2$-cells meet in a single point of the form $\B_{\{\alpha_1,\alpha_2,\alpha_3\}}$. If we consider open cells, then clearly $T$ can be written as the disjoint union 
\[
T= \coprod_{\boldsymbol{\alpha} \subset \mathbb{Z}^3} \mathrm{Int}(\B_{\boldsymbol{\alpha}}).
\]
\smallskip

In our proofs we need new ideas because the criterion for the homeomorphy of a self-affine tile to a $3$-ball established in~\cite{ConnerThuswaldner0000} is applicable only to single tiles, while the theories developed in~\cite{DengLiuNgai18,KLT:15} just cover tiles of a particular shape. Our proofs use the theory of Bing~\cite{Bing51} that leads to a topological characterization of $k$-spheres for $k\le 3$. However, since our conditions differ from the ones of Bing, our proof differs from Bing's proofs and exploits the self-affinity of our tiles.

\smallskip

We have some hope that our theory can be applied to the case $A\ge 2$ as well. However, this generalization would require more case studies and tedious calculations. If negative coefficients $A,B,C$ are permitted, according classes of expanding matrices can be studied. Moreover, Kwun~\cite{Kwun61} and O.~G.~Harrold~\cite{Harrold:61,Harrold:65} establish higher dimensional generalizations of the results of Bing~\cite{Bing51} that we are using here. These results can probably be used to extend our theory to higher dimensions.
 
\subsection{Outline of the paper}
The paper is organized as follows. In Section~\ref{sec:2} we provide preliminaries and basic notions that will be of importance in the proofs of our main results. This includes some graphs that are commonly used in the study of the topology of self-affine tiles and a description of a tiling induced by the truncated octahedron. This tiling is used as a model for the tiling induced by a self-affine tile taken from the class we are interested in. Section~\ref{sec:intersec} describes intersections of subtiles of a self-affine tile. These intersections, which will play an important role in the proofs of our main results, are captured by a large graph, that will be studied in some detail. Finally, Section ~\ref{sec:4} gives an account on the theory of partitionings due to Bing~\cite{Bing51} and defines particular sequences of partitionings that are suitable for our purposes. Finally, these sequences of partitionings are used to establish Theorem~\ref{thm:ball}. Combining Theorem~\ref{thm:ball} with results from \cite{TZ:19} finally leads to the proof of Theorem~\ref{thm:CW}.

\section{Intersections of self-affine tiles and CW complexes}\label{sec:2}

In this section we set up some preliminaries. In Section~\ref{sec:basic} we provide some basic properties of self-affine tiles that will be needed in the sequel. In Section~\ref{sec:normal} we recall that each $3$-dimensional self-affine tile with collinear digit set has a normal form, a so-called {\it $ABC$-tile}. This entails that in the sequel we can restrict ourselves to the investigation of this class of tiles without loss of generality. After that, in Section~\ref{sec:neighbor} we recall the notion of {\em neighbor graph} that permits us to study intersections of the form $T \cap (T+\alpha)$ for an $ABC$-tile $T$. Section~\ref{sec:hata} is devoted to the Hata graph, a graph that surveys the intersections between the sets $T+\alpha$, $\alpha\in \nS$, and gives some results related to this graph. Finally, in Section~\ref{sec:O} we relate an $ABC$-tile $T$ with $14$ neighbors and its lattice tiling to the so-called {\it bitruncated cubic honeycomb}, a lattice tiling of $\R^3$ by {\it truncated octahedra}. 

\subsection{Basic properties of self-affine tiles}\label{sec:basic}
Let $M \in \mathbb{Z}^{3\times 3}$ and $\D \subset \Z^3$ be given in a way that $T=T(M,\D)$ is a self-affine tile. Let
\begin{equation}\label{eq:Di}
\D_i = \D + M\D + \dots + M^{i-1}\D \qquad (i\in \N)
\end{equation}
and define the empty sum $\D_0$ to be equal to the vector $0\in\R^3$. Iterating the set equation \eqref{eq:setequation} for $i\in\N$ times yields
\begin{equation}\label{eq:iset}
T= \bigcup_{d\in \D_i}M^{-i}(T+d). 
\end{equation}
If $\mu$ denotes the Lebesgue measure in $\R^3$ we have
\begin{equation}\label{eq:intersectDi}
\mu((T+d_1) \cap (T+d_2)) = 0 \qquad (d_1,d_2 \in \D_i,\, d_1\neq d_2),
\end{equation}
{\it i.e.}, the sets  in the union on the right hand side of \eqref{eq:iset} are mutually {\em essentially disjoint} ({\it cf.}~\cite[(3.11)]{LagariasWang96a}). For this reason each set of the form $M^{-i}(T+d)$ with $i\in \N$ and $d\in \D_i$ is called a {\em subtile} of $T$. Accordingly, $M^{-k}(\sub+z)$ is called a subtile of  $M^{-k}(T+z)$ if $\sub$ is a subtile of $T$ ($k\in\N$ and $z\in \D_k$).

Because $T$ is a self-affine tile it has nonempty interior. Thus the following is true by \cite[Theorem~1.1]{LagariasWang96a}.
    
\begin{lem}\label{lem:LW1}
A self-affine tile $T$ is equal to the closure of its interior. Its boundary $\partial T$ has zero Lebesgue measure.
\end{lem}

Let $\sub_1,\sub_2$ be two distinct subtiles of $T$. It is clear from the measure disjointness of the union in \eqref{eq:iset} that either $\sub_1 \subset \sub_2$, or $\sub_2 \subset \sub_1$, or $\mu(\sub_1 \cap \sub_2) = 0$. Lemma~\ref{lem:LW1} implies that  
\begin{equation}\label{eq:inteq}
\mu(\sub_1 \cap \sub_2) = 0 \;\Longleftrightarrow\; 
\sub_1^\circ \cap \sub_2^\circ = \emptyset   \;\Longleftrightarrow\; 
\sub_1 \cap \sub_2 = \partial \sub_1 \cap \partial \sub_2.
\end{equation}
In the sequel we will often tacitly make use of these equivalences.

\subsection{A normal form}~\label{sec:normal}
For the tiles of our main results we now define a simple normal form. Let $A,B,C\in\N$ with $1\le A \le B < C$ be given and set
\begin{equation}\label{digit}
M=\begin{pmatrix}
0&0&-C\\
1&0&-B\\
0&1&-A\\
\end{pmatrix} \quad\hbox{and}\quad \D=\left\{\begin{pmatrix}
0\\
0\\
0\\
\end{pmatrix}, \begin{pmatrix}
1\\
0\\
0\\
\end{pmatrix},\dots, \begin{pmatrix}
C-1\\
0\\
0\\
\end{pmatrix}\right\}.
\end{equation}
The matrix $M$ is expanding by \cite[Lemma 2.2]{TZ:19}.  Moreover, $\D$ is a complete set of coset representatives of $\Z^3/M\Z^3$ and $\Z[M,\D]=\Z^3$. Define $T$ by $MT = T + \D$. Then $T$ is a self-affine tile. We call such a tile $T$ an \emph{$ABC$-tile}. We know from \cite[Lemma 2.4]{TZ:19} that an $ABC$-tile $T$ tiles $\R^3$ by $\Z^3$-translates in the sense that $T+\Z^3=\R^3$, where $(T+\alpha_1) \cap (T+\alpha_2)$ has Lebesgue measure $0$ for all $\alpha_1,\alpha_2\in\Z^3$ with $\alpha_1\not=\alpha_2$. We thus say that $\{T+z\,:\, z\in\Z^3\}$ forms a {\em tiling} of $\R^3$.

It turns out that we can confine ourselves to the study of $ABC$-tiles. Indeed, let $M'$ be a $3\times 3$ integer matrix with characteristic polynomial $\chi(x)=x^3+Ax^2+Bx+C$ satisfying $1\le A \le B < C$.  By \cite[Lemma~2.2]{TZ:19} we know that $M'$ is an expanding matrix. Let $v\in \Z^3\setminus\{0\}$ and let $\D'=\{0,v,2v,\dots,(C-1)v\}\subset \Z^3$ be a collinear digit set such that $T'=T'(M',\D')$ is a self-affine tile. Let $T=T(M,\D)$ be the $ABC$-tile with characteristic polynomial $\chi$. From~\cite[Section~2.1]{TZ:19} we know that there is a regular matrix $E$ such that the linear mapping $E:\R^3\to\R^3$ maps $\Z^3$ bijectively onto $\Z[M',\D']$, that $T'=ET$, and that for each $\{\alpha_1,\ldots, \alpha_\ell\} \subset \Z[M',\D']\setminus\{0\}$ we have 
\[
T' \cap (T'+E \alpha_1) \cap \dots \cap (T'+E \alpha_\ell)  = E(T \cap (T +  \alpha_1) \cap \dots \cap (T +\alpha_\ell)).
\] 
It is therefore sufficient to prove Theorems~\ref{thm:ball} and~\ref{thm:CW} for $ABC$-tiles and in all what follows we may restrict our attention to this class of self-affine tiles.  

Let $T=T(M,\D)$ be an $ABC$-tile and let $z\in \mathcal{D}_i$ for some $i\ge 0$. Because $\D$ is a standard digit set there exist unique elements $e_{0},\ldots, e_{i-1}\in \{0,\ldots,C-1\}$ such that  
\[
z=\sum_{j=0}^{i-1} M^j \begin{pmatrix}e_{j} \\0 \\0  \end{pmatrix} .
\]
In this case we write 
\begin{equation}\label{eq:Mnotation}
z=(e_{i-1},\ldots,e_0)_M. 
\end{equation}
This notation will prove particularly useful for $i=1$ to write digits in a space-saving way. 

\subsection{The neighbor graph}\label{sec:neighbor}
Let $T=T(M,\D)$ be an $ABC$-tile. In the sequel we will need the so-called \emph{neighbor graph} (see {\it e.g.}~\cite{ScheicherThuswaldner03}), a graph that can be used to describe the intersections $\B_\alpha=T\cap (T+\alpha)$ for $\alpha\in \nS$. We begin by recalling some definitions from graph theory. For a directed labeled graph $G$ with set of nodes $V$, set of edges $E$, and set of edge-labels $L$, we write an edge leading from  $v \in V$ to $v' \in V$ labeled by $\ell \in L$ as $v\xrightarrow{\ell} v'$. In this case $v$ is called a \emph{predecessor} of $v'$ and $v'$ is called a \emph{successor} of $v$. Following {\it e.g.}~Diestel~\cite[Chapter~1]{Diestel:05}  a (finite or infinite) sequence $v_0\xrightarrow{\ell_1}v_1\xrightarrow{\ell_2}v_2\xrightarrow{\ell_3}\cdots$ of consecutive edges in $G$ is called a \emph{walk}. A walk whose nodes $v_0,v_1,v_2,\ldots$ are mutually distinct is called a \emph{path}.  If $G$ is undirected and not labeled, then an edge of $G$ connecting the nodes $v$ and $v'$ is denoted by $v\,\mbox{---}\,v'$. Walks and paths in $G$ are defined as in the directed case as sequences of consecutive edges with or without possible repetitions, respectively. 
The length of a walk is its number of edges. A walk of length $n$ in $G$ that starts and ends at the same node is called an {\em $n$-cycle} if it contains a path of length $n-1$. $G$ is called {\it connected} if for each pair $(v,v')$ of distinct nodes of $G$ there is a path of the form $v\,\mbox{---}\cdots \mbox{---}\,v'$.

\begin{definition}[{Neighbor graph; {\em cf.}~\cite[Section~2]{ScheicherThuswaldner03}}]\label{Graph}
Let $M\in \Z^{3 \times 3}$ and $\D \subset\Z^3$ be given in a way that
$T=T(M,\D)$ is an $ABC$-tile with neighbor set $\nS$. Define the directed labeled {\em neighbor graph} $G(\nS)$ as follows. The nodes of $G(\nS)$ are the neighbors $\nS$, and there is a labeled edge 
\begin{align}\label{eq:arrow}
\alpha\xrightarrow{d|d'}\alpha'\quad\text{  if and only if  }  M\alpha+d'-d=\alpha'  \text{ with }\alpha,\alpha' \in \nS \text{ and } d, d' \in \mathcal{D}.
\end{align}
\end{definition}

In \eqref{eq:arrow} the vector $d'$ is determined by $\alpha,\alpha',d$. Thus we will often just write $\alpha\xrightarrow{d}\alpha'$ instead of $\alpha\xrightarrow{d|d'}\alpha'$.  The notation $\alpha \in G(\nS)$ means that $\alpha$ is a node of $G(\nS)$ and $\alpha\xrightarrow{d}\alpha' \in G(\nS)$ means that
$\alpha\xrightarrow{d}\alpha'$ is an edge of $G(\nS)$. For walks we will use an analogous notation. 

Let $T=T(M,\D)$ be an $ABC$-tile. Because $\{T+z\;:\; z\in\Z^3 \}$ forms a tiling of $\R^3$, we have
\begin{equation}\label{boun_1}
\partial T=\bigcup_{\alpha\in\mathcal{S}}\B_{\alpha}.
\end{equation} 
Here $\mathcal{S}$ and $\B_{\alpha}$, $\alpha\in\nS$, are given by \eqref{eq:neigh} and \eqref{eq:bgamma}, respectively (note that $\Z[M,\D]=\Z^3$ in these definitions because $T$ is an $ABC$-tile). One can show (see {\it e.g.}~\cite[Proposition~2.2]{ScheicherThuswaldner03}) that the nonempty compact sets $\B_\alpha$, $\alpha\in\nS$, are uniquely determined by the set equations 
\begin{equation}\label{eq:bgammaseteq}
\B_\alpha=\bigcup_{
\begin{subarray}{c}d\in\D, \alpha'\in\nS \\ \alpha\xrightarrow{d} \alpha' \in G(\nS)\end{subarray}}M^{-1}(\B_{\alpha'}+d) \qquad(\alpha\in\nS).
\end{equation}
Here the union on the right hand side of \eqref{eq:bgammaseteq} is extended over all $d,\alpha'$ with $\alpha\xrightarrow{d}\alpha'\in G(\nS)$. The defining equation \eqref{eq:bgammaseteq} is an instance of a {\it graph-directed iterated function system}. These objects were first studied in \cite{MauldinWilliams88}. By \eqref{boun_1} and  \eqref{eq:bgammaseteq} the boundary $\partial T$ is determined by the graph $G(\nS)$. 

\begin{figure}[htbp]
\includegraphics[width=0.8\textwidth]{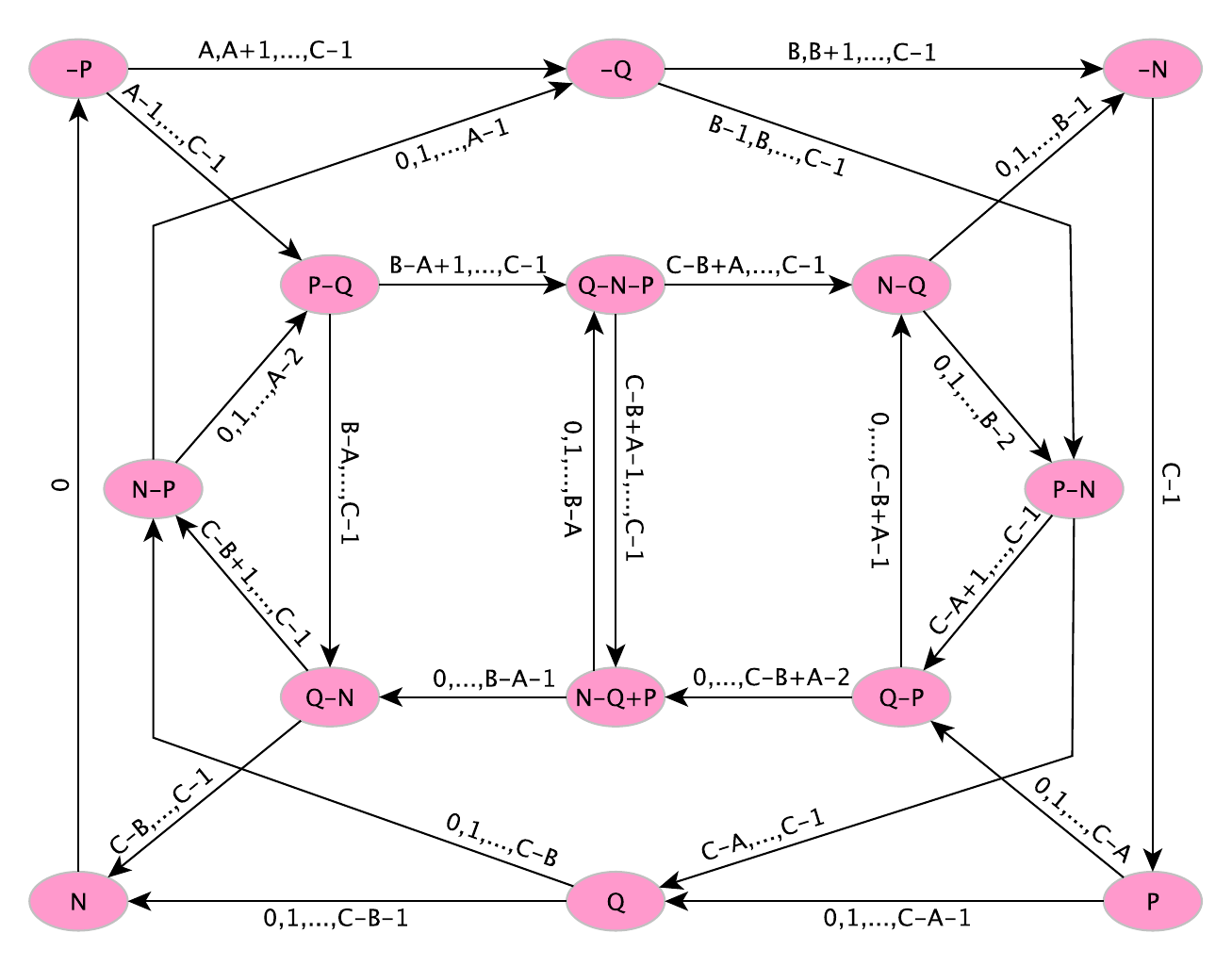}
\caption{
The neighbor graph $G(\nS)$ for an $ABC$-tile $T$ with $1\le A \le B<C$ having $14$ neighbors.
Here we set $P=(1,0,0)^t, ~Q=(A,1,0)^t,~ N=(B,A,1)^t$. To save space we write $\alpha\xrightarrow{e}\alpha'$ instead of $\alpha\xrightarrow{(e)_M}\alpha'$ in this figure (recall the notation \eqref{eq:Mnotation}). Multiple labels correspond to multiple edges. If an edge has labels $d,\ldots, d'$ with $d>d'$ then the edge has to be deleted.
\label{fig:Gamma2_1}}
\end{figure}

The set $\nS$ as well as the neighbor graph $G(\nS)$ of an $ABC$-tile $T=T(M,\D)$ can be calculated explicitly. In the present paper we are interested in $ABC$-tiles having $14$ neighbors  (observe the characterization in Remark~\ref{rem:2}).  In \cite[Section~2.4]{TZ:19} the following results have been proved.
Suppose that $T$ has $14$ neighbors. Then the neighbor set $\nS$ and the neighbor graph $G(\nS)$ are given as follows. Set 
\[
\nS_1=\{P, Q, N, Q-P, N-P, N-Q, N-Q+P\},
\]
where $P=(1,0,0)^t$,  $Q=(A,1,0)^t$, and $N=(B,A,1)^t$. Then the $ABC$-tile $T$ has the neighbors 
$\nS= \nS_1 \cup (-\nS_1)$. Moreover, in this case the neighbor graph $G(\nS)$ is given by the graph in Figure~\ref{fig:Gamma2_1}.

\begin{remark}
This neighbor graph is strongly related to the {\em de Bruijn graph} $N_4$ of binary words of length $4$ (see \cite[Section~3]{deBruijn:46}). Indeed, if we delete the nodes corresponding to the words $0000$ and $1111$ in $N_4$ we get the graph in Figure~\ref{fig:Gamma2_1} (apart from the edge labels). 
\end{remark}

\subsection{The Hata graph and Peano continua}\label{sec:hata}
Let $T=T(M,\D)$ be an $ABC$-tile.
The {\em Hata graph} $H(\nS)$ of the neighbors of $T$ is defined as follows. The nodes of $H(\nS)$ are the elements of $\nS$ and there is an undirected edge between two distinct elements $\alpha_1,\alpha_2 \in \nS$ if and only if $(T+\alpha_1) \cap (T+\alpha_2)\not=\emptyset$. For an $ABC$-tile with $14$ neighbors the Hata graph $H(\nS)$ is depicted in Figure~\ref{HataGraph}. 
It can be determined by using \cite[Lemma~2.16]{TZ:19} (see also \cite[Figure~9]{TZ:19}). The following lemma is a reformulation of some basic results from \cite[Section~2]{TZ:19}.

\begin{figure}[h]
\begin{center}
\includegraphics[width=0.35 \textwidth]{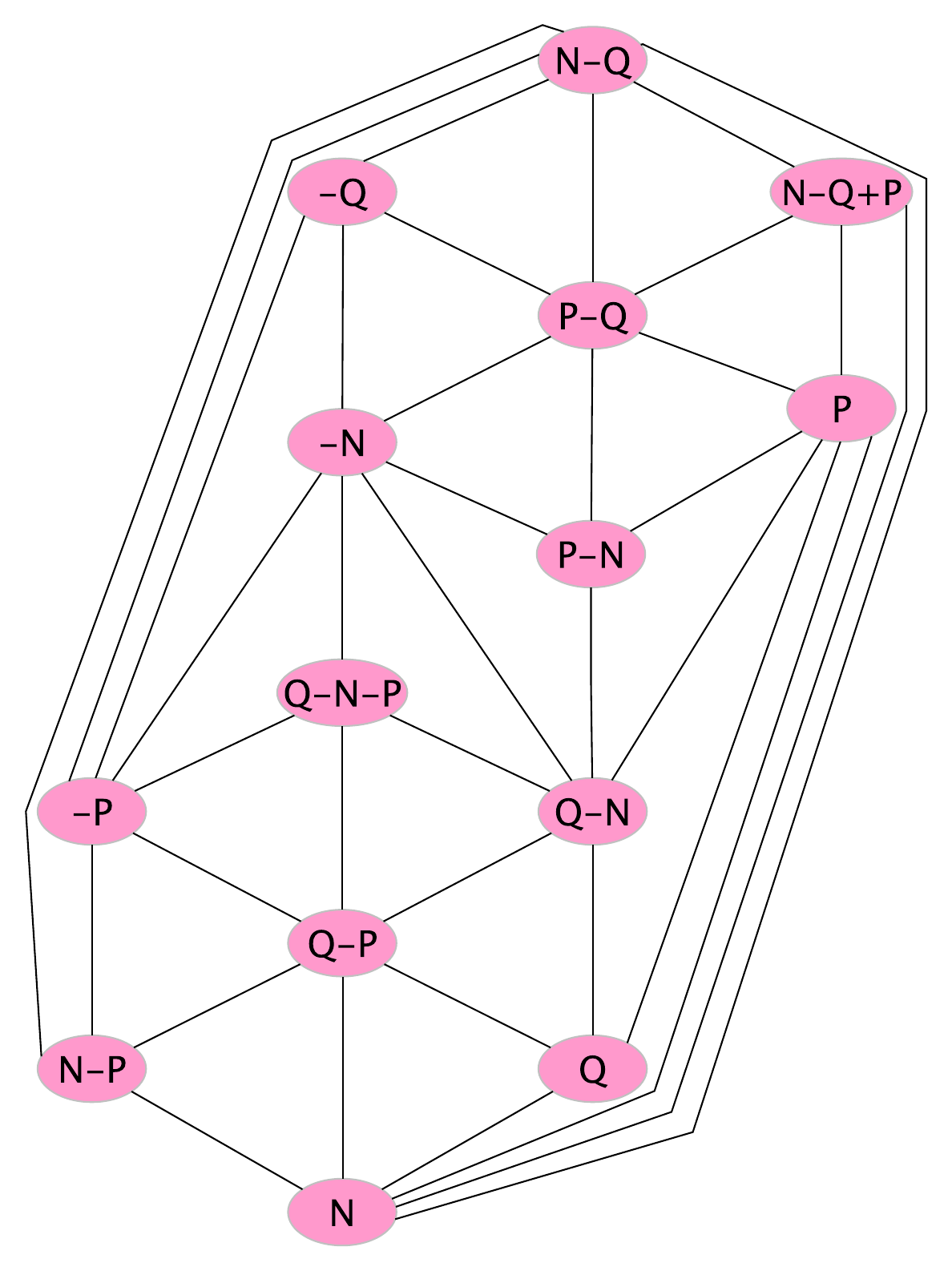} \quad
\includegraphics[trim=0 -60 0 0,width=0.35 \textwidth]{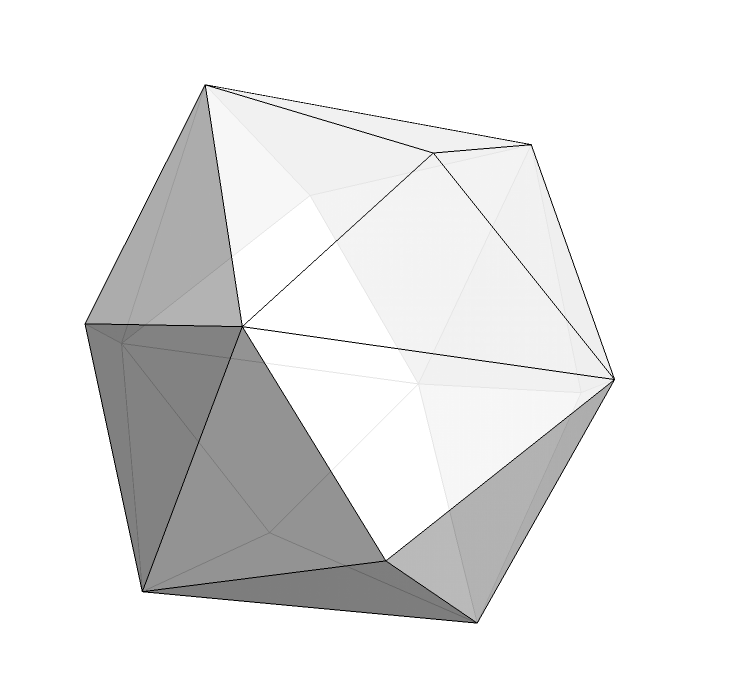}
\end{center}
\caption{The Hata graph $H(\nS)$ (left) which is isomorphic to the graph of vertices and edges of the so-called {\it tetrakis hexahedron}. The tetrakis hexahedron (right) is a {\em Catalan polyhedron} which is the dual of the truncated octahedron (see~{\it e.g.}~\cite[p.~284]{CBG:08}).
\label{HataGraph} } 
\end{figure}

\begin{lem}\label{lem:cycl}
Let $T$ be an $ABC$-tile with $14$ neighbors.  Let $\alpha_1,\alpha_2,\alpha_3 \in \mathbb{Z}^3\setminus\{0\}$ be mutually distinct. We have
\begin{itemize}
\item[(1)]  $\B_{\alpha_1} \not=\emptyset$ if and only if $\alpha_1$ is a node of $H(\nS)$.
\item[(2)]  $\B_{\{\alpha_1,\alpha_2\}}\not=\emptyset$ if and only if $\alpha_1 \relbar\!\relbar\!\relbar \alpha_2$ is an edge in $H(\nS)$.
\item[(3)] $\B_{\{\alpha_1,\alpha_2,\alpha_3\}}\not=\emptyset$  if and only if there is a 3-cycle with nodes $\alpha_1,\alpha_2,\alpha_3$ in $H(\nS)$. 
\item[(4)] If $\boldsymbol{\alpha}\subset\Z^3 \setminus \{0\}$ has more than three elements then $\B_{\boldsymbol{\alpha}} = \emptyset$.
\end{itemize}
\end{lem}

\begin{proof}
Item (1) follows because the nodes of $H(\nS)$ are the neighbors of $T$.
Item (2) follows from the definition of the edges of $H(\nS)$.
Items (3) and (4) follow from \cite[Lemma~2.16]{TZ:19}. For (3) one just has to check that the nodes of the graph $G_3(\nS)$ defined in \cite[Figure~6]{TZ:19} are in one-to-one correspondence with the $3$-cycles of $H(\nS)$.
\end{proof}

The Hata graph $H(\nS)$ and some other Hata graphs are used in the proof of the following lemma.

\begin{lem}\label{lem:HataPeano}
Let $T$ be an $ABC$-tile with $14$ neighbors. Then $T$ and $\partial T$ are Peano continua.
\end{lem}

\begin{proof}
Since $P\in \nS$, we have $M^{-1}\B_P = M^{-1}T\cap M^{-1}(T+P) \not=\emptyset$. Thus $T$ is a Peano continuum by \eqref{eq:setequation} and \cite[Theorem~4.6]{Hata85}. 

Next we prove that $\B_\alpha$ is a Peano continuum for each $\alpha\in\nS$. For $\alpha\in\nS$ let 
\[
Z_\alpha = \{ M^{-1}(\B_{\alpha'}+d) \;:\; d\in\D, \alpha'\in\nS \hbox{ such that } \alpha\xrightarrow{d} \alpha' \in G(\nS) \hbox{ exists}\}
\]
be the collection of the sets in the union on the right hand side of \eqref{eq:bgammaseteq}. The {\em Hata graph} of $Z_\alpha$ is the undirected graph $H_\alpha$ whose nodes are the elements of $Z_\alpha$ and that has an edge between two distinct elements of $b_1,b_2 \in Z_\alpha$ if and only if $b_1\cap b_2 \not=\emptyset$.  According to \cite[Lemma~3.3]{TZ:19} (see also~\cite[Theorem~4.1]{LuoAkiyamaThuswaldner04}), to establish the claim we have to prove that $H_\alpha$ is connected for each $\alpha\in \nS$. To this matter we have to construct the graphs $H_\alpha$. This is done by checking whether intersections of the form $b_1\cap b_2$ with distinct $b_1,b_2 \in Z_\alpha$ are empty or not.  Since 
$b_1=M^{-1}((T+d_1)\cap(T+d_1+\alpha_1))$ and $b_2=M^{-1}((T+d_2)\cap(T+d_2+\alpha_2))$
with some $d_1,d_2\in\D$ and some $\alpha_1,\alpha_2\in \nS$,
\begin{equation}\label{eq:44}
b_1 \cap b_2 = M^{-1}((T+d_1)\cap(T+d_1+\alpha_1)\cap(T+d_2)\cap(T+d_2+\alpha_2)).
\end{equation}
Set $\boldsymbol{\alpha}=\{ \alpha_1, d_2-d_1, \alpha_2+d_2-d_1 \}\setminus\{0\}$. Then
$b_1\cap b_2$ is an affine image of $\B_{\boldsymbol{\alpha}}$,
where $|\boldsymbol{\alpha}|\in\{2,3\}$ depending on whether the four translates $\{d_1,d_1+\alpha_1,d_2,d_2+\alpha_2\}$  in \eqref{eq:44} are mutually distinct or not. But if $\B_{\boldsymbol{\alpha}}$ is empty or not can be read off the Hata graph $H(\nS)$ in view of Lemma~\ref{lem:cycl}. For $\alpha=P$ we see from Figure~\ref{fig:Gamma2_1} that the nodes of $H_P$ are
\[
Z_P =\{M^{-1}(\B_{Q} + (e)_M)\;:\; 0\le e\le C-A-1\} \cup \{M^{-1}(\B_{Q-P} + (e)_M)\;:\; 0\le e\le C-A\}.
\]
Let $b_1,b_2\in Z_P$ be distinct. Inspecting $H(\nS)$  (or directly from \cite[Corollary~3.23]{TZ:19}) we see that the Hata graph $H_P$ is the line given in Figure~\ref{fig:hataalpha}, and, hence, $H_P$ is connected.

\begin{figure}[h]
\includegraphics[width=10cm]{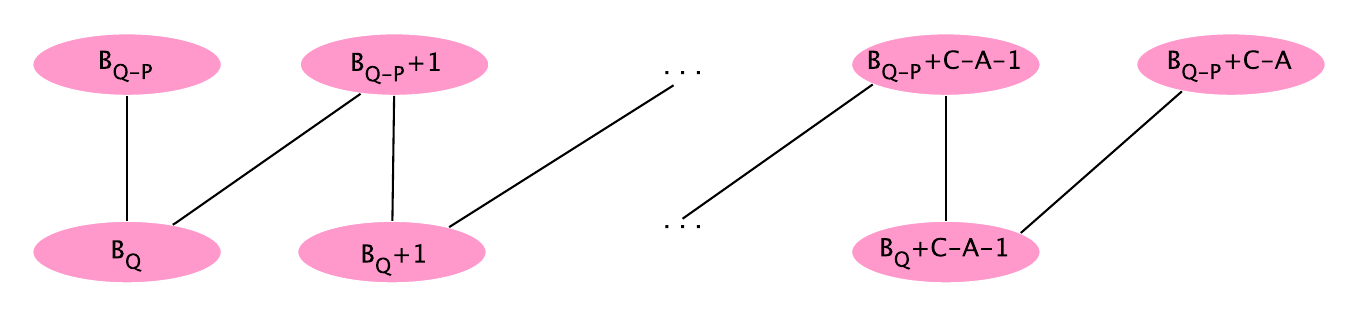}
\caption{The Hata graph $H_P$ (we omit the multiplication by $M^{-1}$ and write $e$ instead of $(e)_M$ to save space). \label{fig:hataalpha}}
\end{figure}

Analogously we see that $H_\alpha$ is a line or a single node and, hence, connected for each $\alpha\in\nS\setminus\{P\}$ as well. Thus \cite[Lemma~3.3]{TZ:19} yields that $\B_\alpha$ is a Peano continuum\footnote{It is easy to see from \eqref{eq:bgammaseteq} that $\B_\alpha$ is not a single point ($\alpha\in\nS$).}
for each $\alpha\in\nS$.

Since $T$ is connected, $\partial T$ is connected as well by \cite[Theorem~1.2]{LuoAkiyamaThuswaldner04}. Therefore, by \eqref{boun_1}, $\partial T$ is a connected union of finitely many Peano continua and, hence, a Peano continuum. 
\end{proof}

The fact that $\partial T$ is a Peano continuum is not used in the present paper. However, it is tacitly used in \cite[Section~3.4]{TZ:19} without giving a formal proof (although in \cite[Corollary~3.23 and Lemma~3.3]{TZ:19} all ingredients for the proof are provided). Thus we decided to prove it here before we state the following version of the main result of \cite{TZ:19}, which is formulated by using $H(\nS)$. 

\begin{prop}\label{lem:TZ:19}
Let $T$ be an $ABC$-tile with $14$ neighbors and let $\alpha_1,\alpha_2,\alpha_3 \in \mathbb{Z}^3\setminus\{0\}$ be mutually distinct. Then the following assertions hold.
\begin{itemize}
\item[(1)]  $\B_{\alpha_1}$ is a $2$-ball if $\alpha_1 \in \nS$, and  $\B_{\alpha_1} =\emptyset$ otherwise.
\item[(2)] $\B_{\{\alpha_1,\alpha_2\}}$ is a $1$-ball if there is an edge $\alpha_1 \relbar\!\relbar\!\relbar \alpha_2$ in $H(\nS)$, and  $\B_{\{\alpha_1,\alpha_2\}}=\emptyset$ otherwise. Moreover, for each $\alpha_1\in \nS$ we have
\[
\bigcup_{\alpha_2: \; \alpha_1 \relbar\!\relbar\!\relbar \alpha_2 \in H(\nS)} \B_{\{\alpha_1,\alpha_2\}}
\simeq \mathbb{S}^1.
\]
\item[(3)] $\B_{\{\alpha_1,\alpha_2,\alpha_3\}}$ is a $0$-ball if there is a 3-cycle $\alpha_1 \relbar\!\relbar\!\relbar \alpha_2  \relbar\!\relbar\!\relbar \alpha_3  \relbar\!\relbar\!\relbar \alpha_1$ in $H(\nS)$, and  $\B_{\{\alpha_1,\alpha_2,\alpha_3\}}=\emptyset$ otherwise. 
\item[(4)] If $\boldsymbol{\alpha}\subset\Z^3 \setminus \{0\}$ has more than three elements then $\B_{\boldsymbol{\alpha}} = \emptyset$.
\end{itemize}
\end{prop}

\begin{proof}
Assertion (1) is the content of \cite[Theorem~1.1~(2)]{TZ:19}.
Assertion~(2) follows from \cite[Proposition~3.10~(2)]{TZ:19} and Lemma~\ref{lem:cycl}~(2).
To see assertion~(3) observe that in \cite[Section~3.1]{TZ:19} it is shown that $\B_{\{\alpha_1,\alpha_2,\alpha_3\}}$ is either a singleton or empty. Thus~(3) follows from Lemma~\ref{lem:cycl}~(3). Assertion~(4) is just Lemma~\ref{lem:cycl}~(4).
\end{proof}

\subsection{On the topology of certain subsets of $\partial T$}\label{sec:O}
Let $M\in \Z^{3 \times 3}$ and $\D \subset\Z^3$ be given in a way that
$T=T(M,\D)$ is an $ABC$-tile. Suppose that $T$ has $14$ neighbors. In what follows we will need precise information on the topology of the subsets
\begin{equation}\label{eq:UR}
U(R) = \bigcup_{\alpha\in R} \B_{\alpha} \qquad(R \subseteq \nS)
\end{equation}
of the boundary $\partial T$.

Let $O$ be a truncated octahedron whose sides are labeled by the elements of $\nS$ in the way shown on the left hand side of Figure~\ref{fig:Honey}, with the convention that the side opposite to the side labeled with $\alpha\in \nS$ is labeled with $-\alpha$. We denote the face of $O$ labeled with $\alpha\in\nS$ by $O_\alpha$. Moreover, for $\boldsymbol{\alpha} \subseteq \nS$ we define the intersections
\begin{equation}\label{eq:Oa}
O_{\boldsymbol{\alpha}} = \bigcap_{\alpha\in \boldsymbol{\alpha}} O_\alpha
\end{equation} 
with the convention that $O_\emptyset=O$.
\begin{figure}[h]
\includegraphics[trim=0 60 0 0, width=5.5cm]{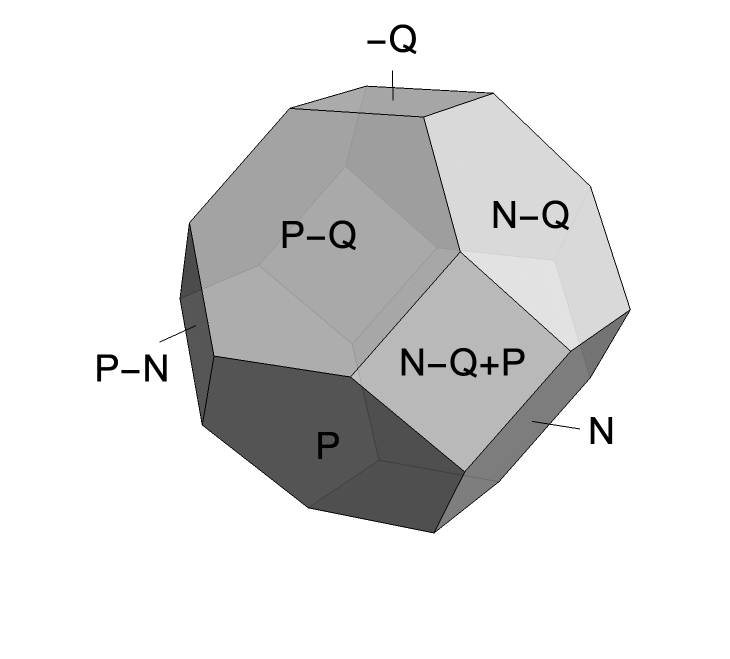}
\includegraphics[trim=0 60 0 0, width=5cm]{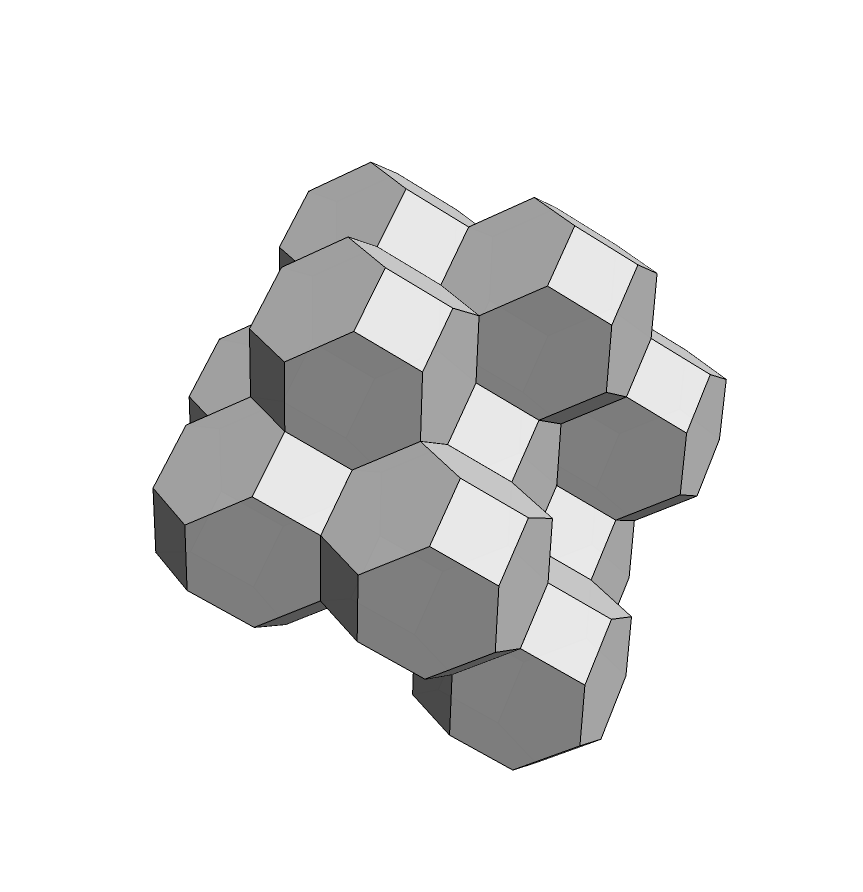}
\caption{A truncated octahedron and a patch of the bitruncated cubic honeycomb. \label{fig:Honey}}
\end{figure}
It is well-known that $O$ induces a tiling of the $3$-dimensional Euclidean space: the so-called {\em bitruncated cubic honeycomb} (see the left hand side of Figure~\ref{fig:Honey} for a patch of this tiling). This tiling has the same ``intersection structure'' as $\{T + z \,:\, z\in \Z^3\}$. In particular, comparing the labeled octahedron $O$ from Figure~\ref{fig:Honey} with Proposition~\ref{lem:TZ:19} we see that the following result holds.

\begin{lem}\label{lem:OB}
Let $T$ be an $ABC$-tile with $14$ neighbors. 
For each nonempty $\boldsymbol{\alpha} \subseteq \nS$ we have
\[
\B_{\boldsymbol{\alpha}} \simeq O_{\boldsymbol{\alpha}}.
\]
\end{lem}
Moreover, we get the following topological characterization of the sets $U(R)$.

\begin{lem}\label{eq:UrO}
Let $T$ be an $ABC$-tile with $14$ neighbors. Let $R \subseteq \nS$ be given. Then
\begin{equation}\label{eq:Uprime}
U(R) \simeq \bigcup_{\alpha\in R} O_{\alpha}.
\end{equation}
Here $U(R)$ is as in \eqref{eq:UR}.
\end{lem} 

\begin{proof}
Denote the right hand side of \eqref{eq:Uprime} by $U'(R)$. It is easy to see that $U'(R)$ is a CW complex\footnote{Again we use closed cells instead of open ones for convenience.} ({\it cf}.~\cite[p.~5]{Hatcher:02}). Indeed, for $i\in\{0,1,2\}$ the closed $i$-cells are given by the nonempty sets $O_{\boldsymbol{\alpha}}$ with $\boldsymbol{\alpha}\subseteq\nS$, $\boldsymbol{\alpha} \cap R \neq \emptyset$, and $\#\boldsymbol{\alpha} = 3-i$. Thus the $0$-skeleton $U'(R)^0$ is the set of vertices of $U'(R)$. Each closed $1$-cell $O_{\{\alpha_1,\alpha_2\}}$ is attached to the two closed $0$-cells $O_{\boldsymbol{\alpha}}$ satisfying $\boldsymbol{\alpha} \supset \{\alpha_1,\alpha_2\}$ and $\#\boldsymbol{\alpha} = 3$. This yields the $1$-skeleton $U'(R)^1$. To get $U'(R)$ we attach each closed $2$-cell $O_{\alpha_1}$, $\alpha_1\in R$, to the circle $\bigcup_{\alpha_2 \in \nS :\alpha_2\not=\alpha_1} O_{\{\alpha_1,\alpha_2\}}$. 

From Proposition~\ref{lem:TZ:19} we see that the set $U(R)$ is a CW complex whose closed $i$-cells are given by the nonempty sets $\B_{\boldsymbol{\alpha}}$ with $\boldsymbol{\alpha}\subseteq\nS$, $\boldsymbol{\alpha} \cap R \neq \emptyset$, and $\#\boldsymbol{\alpha} = 3-i$ for $i\in\{0,1,2\}$ with analogous attaching rules as above. 

Thus, by Lemma~\ref{lem:OB}, $U(R)$ and $U'(R)$ have isomorphic CW complex structures, hence, they are isomorphic as topological spaces.
\end{proof}

This lemma reduces the problem of determining the topology of $U(R)$ to a simple combinatorial problem. In Figure~\ref{fig:ur} we give two examples. The one on the left hand side shows that $U(R)$ is a $2$-ball if $R=\{P,N-Q,N-Q+P\}$, from the second one we immediately see that $U(R)$ is the  union of $2$ disjoint $2$-balls if $R=\{N,N-P,N-Q,N-Q+P,Q-N-P\}$.
\begin{figure}[h]
\includegraphics[trim=0 50 0 0, width=4.7cm]{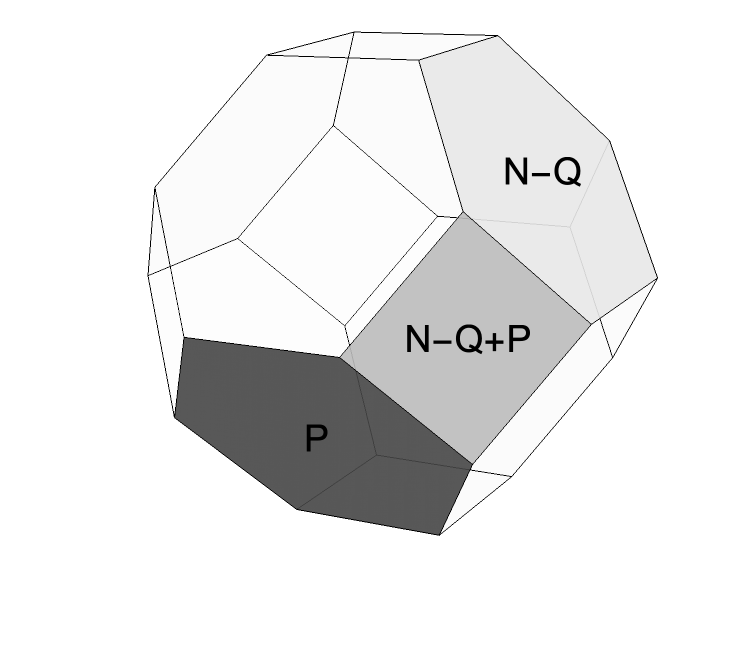}
\includegraphics[trim=0 60 0 0, width=5cm]{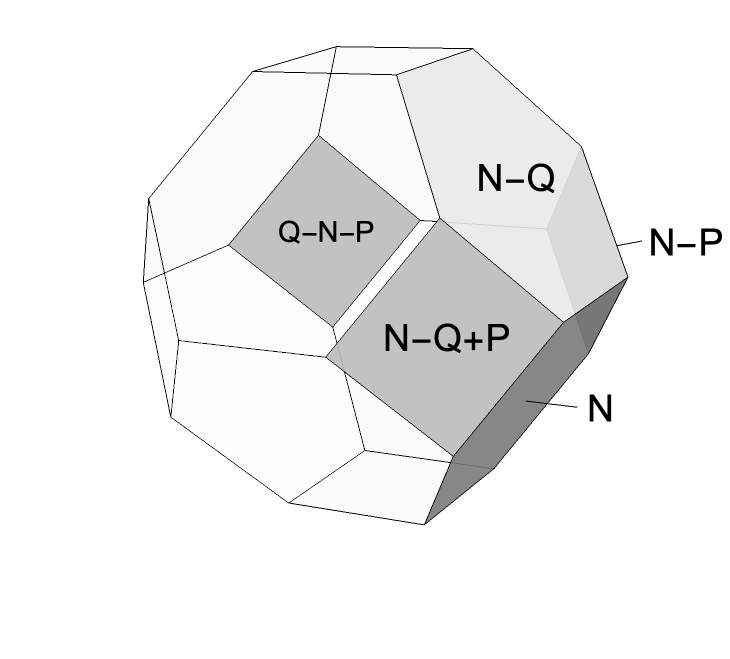}
\caption{The set $\bigcup_{\alpha\in R} O_\alpha$ for two choices of $R\subseteq \nS$. \label{fig:ur}}
\end{figure}

\section{Types of intersections}\label{sec:intersec}
Let $T=T(M,\D)$ be an $ABC$-tile with $14$ neighbors. In Section~\ref{sec:is} we study basic properties of intersections of the form $\sub_1\cap \sub_2$ where $\sub_1$ and $\sub_2$ are essentially disjoint subtiles of $T$ ($t_1$ may also be equal to $\overline{\R^3 \setminus T}$). We will show that we can attach to $\sub_1\cap \sub_2$ a set $R\subseteq \nS$ such that $\sub_1\cap \sub_2 \simeq U(R)$. According to Lemma~\ref{eq:UrO}, the topology of $U(R)$ is easy to determine. Knowing the topology of such intersections will be important in order to apply the results of Bing~\cite{Bing51} that will be needed in the proof of Theorem~\ref{thm:ball}. Section~\ref{sec:type} shows a way to choose the set $R \subseteq \nS$ for each intersection $\sub_1\cap \sub_2$ in a unique way (up to sign changes). This set is, by definition, the {\em type} of the intersection. In Section~\ref{sec:isgraph} we define a graph that will help us to survey the possible types of intersection ({\it i.e.}, the possible subsets $R$) that will occur in this context.

\subsection{Basic properties of intersections}\label{sec:is}
The definition of the type of an intersection requires some preparation. Let $T=T(M,\D)$ be an $ABC$-tile. 
Let 
\[
\sub_\infty = \overline{\R^3 \setminus T}=T+(\Z^3\setminus\{0\})
\]
be the closure of the complement of $T$. We define the collection (recall that $\D_i$ is defined in~\eqref{eq:Di})
\[
\nC= \{M^{-i}(T+d)\; : \;  i\in \N,\, d\in \D_i   \} \cup \{\sub_\infty\}
\]
that contains $\sub_\infty$ as well as each of the subtiles of $T$. If $\sub \in \nC$ we define
\begin{equation}\label{eq:leveldef}
\level{t} = \begin{cases}
i, & \hbox{if $\sub$ is of the form $M^{-i}(T+d)$ for $i\in \N$ and $d\in \D_i$},  \\
-\infty, & \hbox{if $\sub=\sub_\infty$}.
\end{cases}
\end{equation}
We provide the following simple result. Recall that $U(R)$ is defined in \eqref{eq:UR}.

\begin{lem}\label{lem:Rchar}
Let $T$ be an $ABC$-tile with $14$ neighbors.  Let $\sub_1,\sub_2 \in \nC$ be essentially disjoint. Then there is $R \subseteq \nS$ (possibly empty) such that 
$\sub_1 \cap \sub_2 = M^{-\ell}(U(R)+ d)$
for some $\ell\in\N$ and some $d\in \Z^3$.
\end{lem}
 
 \begin{proof}
 Assume w.l.o.g.\ that $\level{\sub_1} \le \level{\sub_2}$. Set $\ell_i=\level{\sub_i}$. Then $\ell_2\in\N$ and $\sub_2=M^{-\ell_2}(T+d)$ for some $d\in \D_{\ell_2}$ and, by possibly subdividing $\sub_1$, we see that $\sub_1$ is a union of sets of the form $M^{-\ell_2}(T+z_k)$ with $z_k\in\Z^3\setminus\{d\}$ (this union is infinite if and only if $\sub_1=\sub_\infty$). Thus 
\[
\sub_1 \cap \sub_2 =  \bigcup_{k} M^{-\ell_2}(T+z_k) \cap M^{-\ell_2}(T+d)
= M^{-\ell_2}\bigcup_{k}(\B_{z_k-d} + d).
\]
Because $\B_\alpha \neq \emptyset$ holds if and only if $\alpha \in \nS$ there is a set $R\subseteq \nS$ such that
\begin{equation*}\label{eq:g1g2cap}
\sub_1 \cap \sub_2 = M^{-\ell_2}\bigcup_{\alpha \in R}(\B_{\alpha}+ d)=M^{-\ell_2}(U(R)+ d). \qedhere
\end{equation*}
 \end{proof}
 
By this lemma the topology of the intersection of two essentially disjoint elements of $\nC$ can be described in terms of a subset $R\subseteq\nS$. Using the notation \eqref{eq:Mnotation}, from \eqref{eq:arrow} we gain
 \begin{equation}\label{eq:GSsymm1}
\alpha\xrightarrow{d} \alpha'\in G(\nS) \quad \hbox{if and only if}\quad -\alpha
\xrightarrow{(C-1)_M-d} -\alpha'\in G(\nS). 
 \end{equation}
Thus \eqref{eq:bgammaseteq} yields $\B_{-\alpha}=x_C-\B_{\alpha}$ with
$x_C=\sum_{i\ge 1}M^{-i}(C-1)_M$
for each $\alpha\in\nS$ and, hence, $U(-R) = x_C-U(R)$. 
This implies that $U(-R) \simeq U(R)$, and we therefore want to identify $R$ with $-R$ in this description. To this matter we define the equivalence relation $\approx$ on the power set $2^\nS$ of $\nS$ by $R\approx R'$ if and only if $R'= \pm R$. The equivalence classes of this relation are denoted by $\overline{R}$ for $R\subseteq \nS$. Since this notation is only used for (finite) subsets $R$ of $\nS$, there is no risk of confusion with the closure $\overline{X}$ of a set $X$, for which the same notation is used.

%


\begin{remark}
Let $\sub_1,\sub_2 \in \nC$ be essentially disjoint. By Lemma~\ref{lem:Rchar},  $\sub_1 \cap \sub_2 = M^{-\ell}(U(R)+ d)$ for some $\ell\in\N$, $d\in \Z^3$, and $R\subset\nS$. We could define $\overline{R}$ as the {\em type} of the intersection of $\sub_1 \cap \sub_2$. However, {\em a priori} $\overline{R}$ is not uniquely defined by this equality and we would have to prove unicity. To circumvent this, in Section~\ref{sec:type} we give another (equivalent) definition of {\em type} that is obviously unique and better suited to our purposes. Roughly speaking, we pick the ``right'' class $\overline{R}$ by using the neighbor graph. The additional effort we need in order to state this definition will pay off later. 
\end{remark}

Before we can define the type of an intersection, we need one more lemma.

\begin{lem}\label{lem:GStypedescription}
Let $T$ be an $ABC$-tile. Let $\alpha \in \Z^3\setminus\{0\}$, $i \ge 0$, and $d=d_i+Md_{i-1}+\dots +M^{i-1}d_1\in \D_i$. Then
\begin{equation}\label{cell-ia}
\begin{split}
(T + \alpha) \cap M^{-i}(T+d)
&=M^{-i}
\bigcup_{\alpha_i:\; \alpha\xrightarrow{d_1}\alpha_1\xrightarrow{d_2}\cdots\xrightarrow{d_i}\alpha_i  \in G(\nS)} (\B_{\alpha_i}+d),
\end{split}
\end{equation}
where the union is extended over all $\alpha_i\in\nS$ for which there exist $\alpha_1,\ldots,\alpha_{i-1}\in\nS$ such that there is a walk $\alpha\xrightarrow{d_1}\alpha_1\xrightarrow{d_2}\cdots\xrightarrow{d_i}\alpha_i\in G(\nS)$. 
\end{lem}

Note that the union in \eqref{cell-ia} may well be empty. This is certainly the case if $\alpha\not\in\nS$.

\begin{proof}
For $i=0$ we have $d=0$ and \eqref{cell-ia}  is trivial. For $i\ge1$ we prove \eqref{cell-ia} by induction on $i$. For the induction start let $i=1$ and observe that for each fixed $d\in \D$ we get, by the set equation \eqref{eq:setequation} and the definition of the edges in $G(\nS)$ provided in \eqref{eq:arrow},
\begin{equation}\label{cell-1a}
\begin{split}
(T + \alpha) \cap M^{-1}(T+d)&=  M^{-1}\big((MT + M\alpha) \cap(T+d)\big)
\\
&= M^{-1} 
\bigcup_{d'\in \D} \big((T + d' + M\alpha) \cap (T+d)
\big) \\
&= M^{-1}
\bigcup_{d'\in \D} \big(\big((T + M\alpha+ d' - d) \cap T\big) + d \big)
\\
&=M^{-1}\bigcup_{\alpha':\; \alpha\xrightarrow{d}\alpha' \in G(\nS)} (\B_{\alpha'}+d).
\end{split}
\end{equation}
For the induction step assume that \eqref{cell-ia} holds for $i-1$ instead of $i$,
let $d'=d_{i-1}+Md_{i-2}+\dots +M^{i-2}d_1\in \D_{i-1}$ and $d=d_i+Md'$. The set equation \eqref{eq:setequation} implies that $M^{-i}(T+d) \subset M^{-i+1}(T+d')$. Thus by the induction hypothesis
\begin{equation*}
\begin{split}
(T + \alpha) \,\cap\,& M^{-i}(T+d)
= (T + \alpha) \cap M^{-i+1}(T+d')  \cap M^{-i}(T+d)  \\
&=M^{-i+1}
\bigcup_{\alpha_{i-1}:\; \alpha\xrightarrow{d_1}\alpha_1\xrightarrow{d_2}\cdots\xrightarrow{d_{i-1}}\alpha_{i-1}  \in G(\nS)} \big((\B_{\alpha_{i-1}}+d')\cap M^{-1}(T+d)\big)
 \\
&=M^{-i+1}
\bigcup_{\alpha_{i-1}:\; \alpha\xrightarrow{d_1}\alpha_1\xrightarrow{d_2}\cdots\xrightarrow{d_{i-1}}\alpha_{i-1}  \in G(\nS)} \big(\big((T+\alpha_{i-1}) \cap M^{-1}(T+d_i)\big)+d'\big).
\end{split}
\end{equation*}
Applying \eqref{cell-1a} to the last intersection yields \eqref{cell-ia} and the induction is finished.
\end{proof}

\subsection{The type of an intersection}\label{sec:type}
We are now ready to define the type of an intersection. Let $T$ be an $ABC$-tile with $14$ neighbors, let $\sub\in \nC\setminus\{\sub_\infty\}$, and set $i=\level{\sub}$. Then there is $d=d_i+Md_{i-1}+\dots +M^{i-1}d_1\in \D_i$ such that $\sub=M^{-i}(T+d) \subseteq T$. Thus Lemma~\ref{lem:GStypedescription} implies that
\begin{equation}\label{cell-i}
\sub_\infty \cap \sub 
= \bigcup_{\alpha\in \nS} ((T+\alpha) \cap M^{-i}(T+d))
=M^{-i}
\bigcup_{\alpha\in \nS}\bigcup_{\alpha_i\;:\alpha\xrightarrow{d_1}\alpha_1\xrightarrow{d_2}\cdots\xrightarrow{d_i}\alpha_i\in G(\nS)} (\B_{\alpha_i}+d).
\end{equation}
We say that the intersection $\sub_\infty \cap \sub$ is of {\em type} $\overline{R(\sub_\infty,\sub)}$ with 
\begin{equation}\label{eq:Rgginf}
R(\sub_\infty,\sub)=\{\alpha_i \;:\, \hbox{there is $\alpha\in \nS$ with } \alpha\xrightarrow{d_1}\alpha_1\xrightarrow{d_2}\cdots\xrightarrow{d_i}\alpha_i\in G(\nS) \}.
\end{equation}
Note that \eqref{cell-i} implies that $\sub_\infty \cap \sub \simeq U(R(\sub_\infty,\sub)) \simeq U(-R(\sub_\infty,\sub))$. Thus the type $\overline{R(\sub_\infty,\sub)}$ determines the topology of the intersection $\sub_\infty \cap \sub$.

Let  $\sub_1,\sub_2\in \nC\setminus\{\sub_\infty\}$ be essentially disjoint and ordered such that  $i=\level{\sub_1}\le \level{\sub_2}=j$.
We can uniquely choose $z\in \Z^3$, $\alpha\in\Z^3\setminus\{0\}$, $d= d_{j-i}+Md_{j-i-1}+\dots +M^{j-i-1}d_1 \in \D_{j-i}$ in a way that
\[
M^i(\sub_1 \cap \sub_2) + z = (T + \alpha) \cap M^{i-j}(T+d).
\] 
Thus Lemma~\ref{lem:GStypedescription} implies that
\begin{equation}\label{eq:gg12}
\sub_1 \cap \sub_2 = M^{-j} \Big(
\bigcup_{\alpha_{j-i}:\, \alpha\xrightarrow{d_{1}}\alpha_{1}\xrightarrow{d_{2}}\cdots\xrightarrow{d_{j-i}}\alpha_{j-i}\in G(\nS)} (\B_{\alpha_{j-i}}+d) \Big)
 - M^{-i}z.
\end{equation}
We say that the intersection $\sub_1 \cap \sub_2$ is of {\em type}\footnote{If $i=j$ we could switch the roles of $\sub_1$ and $\sub_2$. But since it is easy to see that in this case $R(\sub_1,\sub_2)=-R(\sub_2,\sub_1)$, the type $\overline{R(\sub_1,\sub_2)}$ is well defined also in this case.}
 $\overline{R(\sub_1,\sub_2)}$ with 
\begin{equation}\label{eq:gg13}
R(\sub_1,\sub_2)=\{\alpha_{j-i} \;:\; \alpha\xrightarrow{d_{1}}\alpha_{1}\xrightarrow{d_{2}}\cdots\xrightarrow{d_{j-i}}\alpha_{j-i} \in G(\nS)\}.
\end{equation}
Note that \eqref{eq:gg12} implies that $\sub_1 \cap \sub_2 \simeq U(R(\sub_1,\sub_2)) \simeq U(-R(\sub_1,\sub_2))$. Thus the type $\overline{R(\sub_1,\sub_2)}$ determines the topology of the intersection $\sub_1 \cap \sub_2$. Summing up we have the following lemma.

\begin{lem}\label{lem:typetopo}
Let $\sub_1,\sub_2\in\nC$ be essentially disjoint. If $\sub_1\cap \sub_2$ is of type $\overline{R}$ for some $R\subseteq\nS$ then $\sub_1\cap \sub_2 \simeq U(R)$. 
\end{lem}

Let $\sub_1, \sub_2\in \nC$ be essentially disjoint.  If $\sub_1\cap \sub_2$ has a certain type, we want to know how this influences the type of $\sub_1\cap \sub_2'$ for $\sub_2' \in  \nC$ with $\sub_2' \subset \sub_2$. This will be studied in the next section.

\subsection{A graph that governs the types of intersections}\label{sec:isgraph}
Let $T$ be an $ABC$-tile with $14$ neighbors. We want to know which classes $\overline{R}$ are needed to describe all possible intersections of essentially disjoint elements of $\nC$. 
To this end we introduce the following notation. For a subset $R\subseteq \nS$ and a digit $d\in \D$, we define
\begin{equation}\label{eq:ndrDef}
n_d(R):=\{\alpha'\;:\; \alpha\xrightarrow{d} \alpha'\in G(\nS) \text{ for } \alpha  \in R\}.
\end{equation}
Then $n_d(R)$ contains the successors of elements of $R$ in the neighbor graph that can be reached by an edge with label $d$. Of course, $n_d(R)$ is a subset of $\nS$.  By the symmetry property \eqref{eq:GSsymm1} we have 
\begin{equation}\label{eq:GSsymm}
n_d(R) = -n_{(C-1)_M-d}(-R) \qquad(R\subseteq \nS,\, d\in \mathcal{D}).
\end{equation}
Let $N_0=\{\overline{\nS}\}$ be the set containing the residue class of the full set of neighbors and recursively define a nested sequence $(N_k)_{k\geq 0}$ of subsets of the power set $2^\nS$  by
\begin{equation}\label{eq:Nrec}
N_k=\{ \overline{n_d(R)}\;:\;\overline{R}\in N_{k-1}, d\in \D\}\cup N_{k-1} \qquad (k\geq 1).
\end{equation}
By \eqref{eq:GSsymm}, $N_k$ is well-defined because nothing changes if we replace $R$ by $-R$ in the argument of $n_d$ on the right hand side of \eqref{eq:Nrec}. Because $2^\nS$ is finite there exists a minimal $k_0\in \N$ such that $N_{k_0+1}=N_{k_0}$ and, hence, $N_k=N_{k_0}$ for each $k\geq k_0$. This leads to the following definition.

\begin{definition}[Intersection graph]\label{def:nI}
Let $T$ be an $ABC$-tile with $14$ neighbors. 
The \emph{intersection graph} $\nI$ is the graph whose nodes are the elements of\footnote{We leave away the empty set for practical reasons. It would cause many additional edges in the intersection graph.} $N_{k_0}\setminus\{\overline{\emptyset}\}$. And whose edges are defined by
\begin{equation}\label{eq:Nedge}
\overline{R} \rightarrow \overline{R'} \in \nI \quad\hbox{ if and only if }\quad R'= \pm n_d(R) \hbox{ for some } d\in\D
\end{equation}
(which is again well-defined because of \eqref{eq:GSsymm}).
\end{definition}
We note that the {\em in-out graph} defined in~\cite[Section~7]{ConnerThuswaldner0000} is used for a similar purpose as our intersection graph $\nI$. However, $\nI$ has a simpler structure than the in-out graph.

\begin{lem}
Let $T$ be an $ABC$-tile with $14$ neighbors and assume that  $A=1$. Then we have the following two cases for $\nI$.
\begin{itemize}
\item[(1)] For $A=1$, $B=2$, and $C\geq 4$ the graph $\nI$ is given by Figure~\ref{Fig:B=2}. In particular, we have  $\# \nI=55$. 
\item[(2)] For $A=1$, $B\geq 3$, and  $C\geq 2B$ the graph $\nI$ is given by Figure~\ref{Fig:B>2}. In particular, we have $ \# \nI=57$. 
\end{itemize}
By Remark~\ref{rem:2} the constellations $A,B,C$ covered in (1) and (2) exhaust all $ABC$-tiles with $14$ neighbors having $A=1$.
\end{lem}

\begin{remark}\label{rem:graphremark}
The graphs $\nI$ are rather large. Thus we cannot draw them directly. Figure~\ref{Fig:B=2} contains a tree. The quotient graph we obtain by identifying nodes with the same node-label in this tree equals $\nI$ for $A=1,\; B=2,\; C \ge 4$. Similarly if we quotient the tree in Figure~\ref{Fig:B>2} by identifying nodes with the same node-label we obtain $\nI$ for $A=1,\; B\geq 3,\; C\ge 2B$.
\end{remark}

\begin{figure}
\centering
\includegraphics[angle=90,trim={0 0 0 -.5cm},clip=true,height=\textheight]{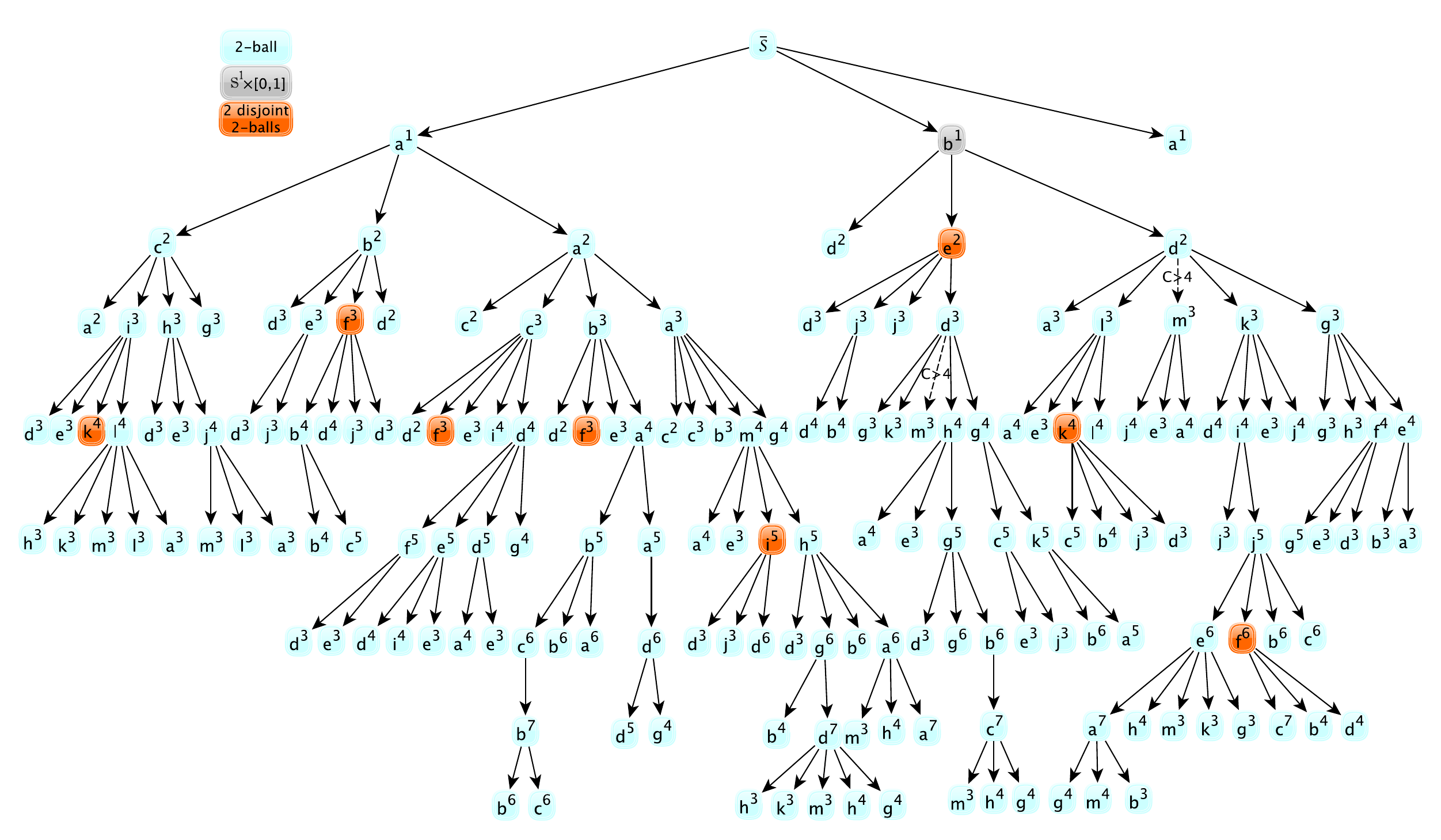}
\caption{The graph $\nI$ for $A=1,B=2, C\geq 4$ is obtained as a quotient graph of this tree; see Remark~\ref{rem:graphremark}.}\label{Fig:B=2}
\end{figure}

\begin{figure} 
\centering
\includegraphics[angle=90,trim={0 0 0 -.5cm},clip=true,height=\textheight]{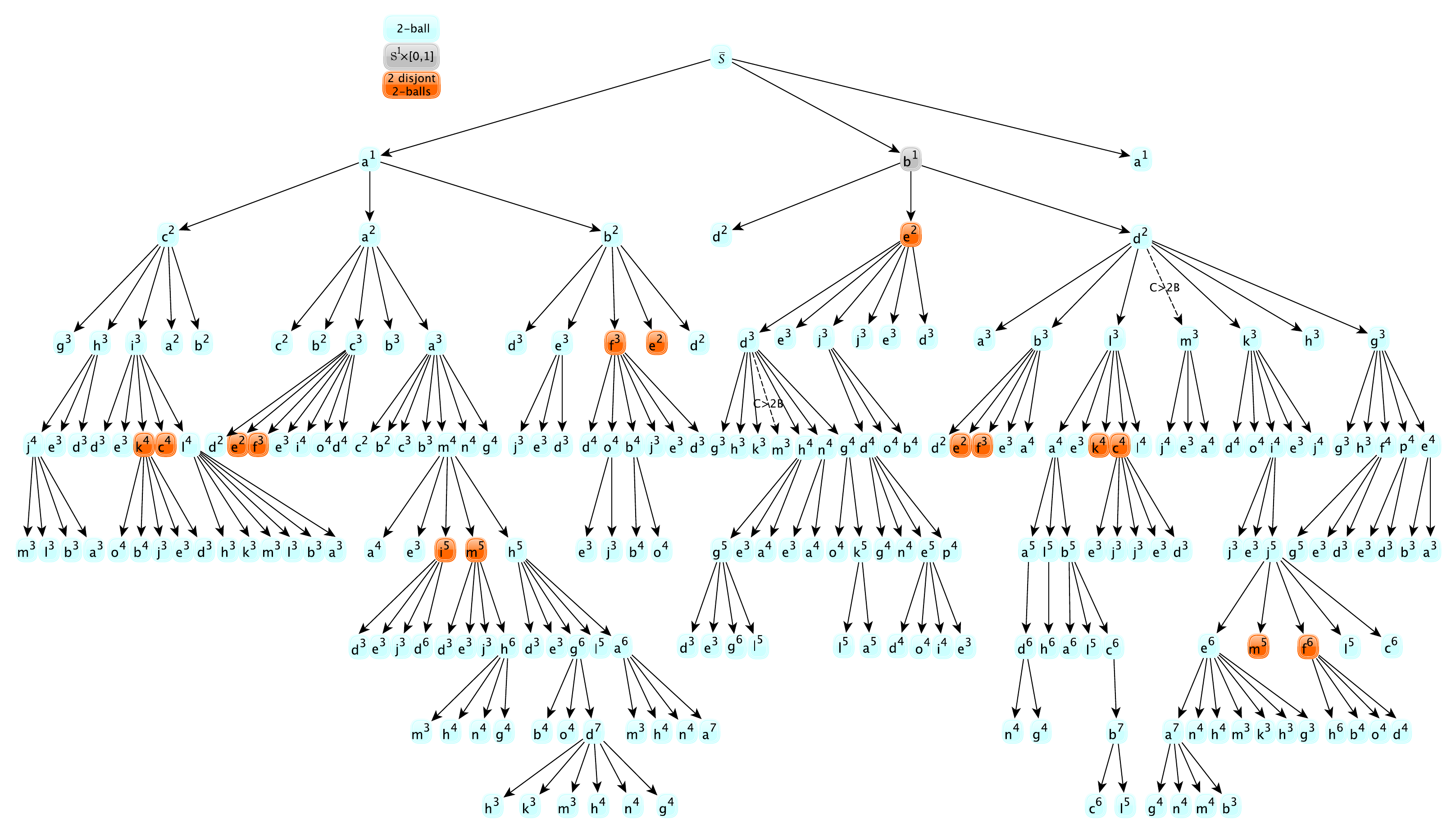}
\caption{The graph $\nI$ for $A=1,B\geq 3, C\geq 2B$ is obtained as a quotient graph of this tree; see Remark~\ref{rem:graphremark}.}\label{Fig:B>2}
\end{figure}

\begin{proof}
This proof is just a lengthy but easy calculation.
Since $N_1=N_1' \cup N_0$ where $N_1'=\{\overline{n_d(\nS)}\,:\,d\in \D\}$, we have to determine $n_d(\nS)$ for each $d\in\D$. Recall the notation \eqref{eq:Mnotation}. From the neighbor graph we see that, for $d=0$, there exists an edge of the form $\alpha \xrightarrow{0|d'}\alpha'$ for each $\alpha' \in \nS\setminus\{P\}$. Thus $\overline{\nS\setminus\{P\}} \in N_1'$. For $d=(e)_M$ with $1\le e \le C-2$, an edge of the form $\alpha \xrightarrow{(e)_M|d'}\alpha'$ exists for  each $\alpha' \in \nS\setminus\{{P,-P}\}$, hence, $\overline{\nS\setminus\{P,-P\}} \in N_1'$. Finally, for $d=(C-1)_M$ an edge of the form $\alpha \xrightarrow{(C-1)_M|d'}\alpha'$ exists for  each $\alpha' \in \nS\setminus\{-P\}$. Thus also $\overline{\nS\setminus\{-P\}}$ is an element of $N_1'$. 
However, since $\overline{\nS\setminus\{-P\}} = \overline{\nS\setminus\{P\}}$ we already got this element before. The sets $\overline{R}$ with $R\subseteq \nS$ contained in $N_1'$ are listed in the second column of Table~\ref{tab_N1}. Table~\ref{tab_N1}, as well as all the other tables\footnote{We provide all these tables in order to illustrate the proof and because we need them for later reference.} in this proof, has the following columns: The first column contains the name of the node $\overline{R}$ in the graphs in Figures~\ref{Fig:B=2} and~\ref{Fig:B>2} corresponding to the subset $R\subseteq \nS$ in the second column. The third column indicates the condition under which this subset occurs. Finally, the fourth column describes the topology of $U(R)$. Recall that according to Lemma~\ref{eq:UrO} the topology of $U(R)$ can be obtained by easy combinatorial arguments which can (as we did) be easily checked by a computer program. Summing up we have shown that 
$
N_1=N_1' \cup N_0 = \{\overline{\nS}, \overline{\nS\setminus\{{P}\}}, \overline{\nS\setminus\{{P,-P}\}} \}
$.
{
\small
\renewcommand{\arraystretch}{1.6} 
\smallskip
\begin{longtable}{|c|c|c|c|}
\hline 
Node&Subset $R$&Condition&Topology of $U(R)$\\
\hline
\circled{$a^1$}&$ \nS\setminus\{{P}\}$&---&$2$-ball\\
\hline 
\circled{$b^1$}&$\nS\setminus\{P,-P\}$&---& $\mathbb{S}^1\times[0,1]$ (a ``ribbon'')\\
\hline 
\caption{The set $N_1'$. \label{tab_N1}}
\end{longtable}
}
Now we can calculate $N_k$ for $k\ge 1$ in an analogous way as follows.
{
\small
\renewcommand{\arraystretch}{1.6} 
\smallskip
\begin{longtable}{|c|c|c|c|}
\hline 
Node&Subset $R$&Condition&Topology of $U(R)$
\endhead
\hline 
\circled{$a^2$}&$\nS\setminus \{P, Q, Q-P\}$ &---&  $2$-ball  \\
\hline 
 \circled{$b^2$}&$\nS\setminus\{P, Q, Q-P,-P\}$&---&$2$-ball  \\
\hline
 \circled{$c^2$}&$\nS\setminus\{P, P-Q\}$&---& $2$-ball \\
\hline
 \circled{$d^2$}&$\nS\setminus\{P,Q,Q-P,P-Q\}$&---& $2$-ball \\
\hline
 \circled{$e^2$}&$\nS\setminus\{P,Q,Q-P,-P,-Q,P-Q\}$&---&
 $2$ disjoint $2$-balls  \\
\hline
\caption{The set $N_2'$ consists of the subsets in the second column in this table. \label{tab:N2}}
\end{longtable}
}
\noindent
Starting from $N_1$ we use \eqref{eq:Nrec} and the neighbor graph $G(\nS)$ to calculate $N_2$. This yields that $N_2=N_2'\cup N_1$ where the set $N_2'$ corresponds to the subsets indicated in Table~\ref{tab:N2}.
{
\small
\renewcommand{\arraystretch}{1.6} 
\smallskip
\begin{longtable}{|c|c|c|c|}
\hline 
Node&Subset $R$&Condition&Topology of $U(R)$
\endhead
\hline
\circled{$a^3$}&$\nS_1$&---&$2$-ball\\
\hline 
\circled{$b^3$}&$\nS_1\setminus \{P\}$ &---&$2$-ball\\
\hline 
\circled{$c^3$}&$(\nS_1\setminus\{P\})\cup\{-N,Q-N-P\}$
&---&$2$-ball\\
 \hline 
\circled{$d^3$}&$\nS_1\setminus\{Q-P\}$ &---& $2$-ball\\
 \hline 
 \circled{$e^3$}&$\nS_1\setminus\{P,Q,Q-P\}$&---&$2$-ball\\
  \hline 
\circled{$f^3$}&$(\nS_1\setminus\{P,Q,Q-P\})\cup\{-N,Q-N-P\}$
 &---&
 $2$ disjoint $2$-balls\\ 
 \hline 
 \circled{$g^3$}&$\{Q,N,Q-P,N-P\}\cup(-\nS_1\setminus\{P-Q\})$&---&$2$-ball\\
 \hline 
 \circled{$h^3$}&$\{Q,N,Q-P,N-P\}\cup(-\nS_1\setminus\{-P,-Q,P-Q\})
$&---&$2$-ball \\ 
\hline 
\circled{$i^3$}&$(\nS_1\setminus\{P,N-Q\})\cup
(-\nS_1\setminus\{-P,-Q,P-Q\})$
&---&$2$-ball\\
\hline 
\circled{$j^3$}&$\{N,N-Q+P\}$ &---&$2$-ball\\
\hline 
\circled{$k^3$}&$\{Q, N, Q-P, N-P, -N, Q-N-P\}$&---&$2$-ball\\
\hline 
\circled{$l^3$}&$\nS_1\setminus\{P,N-Q\}$&---& $2$-ball\\
\hline 
\circled{$m^3$} & $\{Q, N, Q-P, N-P\}$&$C>2B$&$2$-ball\\
\hline 
\caption{The set $N_3'$. \label{tab:N3}}
\end{longtable}
}
\noindent
We now go on in the same way. If $N_3'$ consists of the sets in the second column of Table~\ref{tab:N3} then, using \eqref{eq:Nrec}, a somewhat lengthy but easy calculation shows that $N_3=N_3'\cup N_2$. Here we have to be careful about the node \circled{$m^3$}. This node only occurs in $N_3'$ if $C>2B$. If $C = 2B$, it occurs in $N_5'$. Thus we have $\# N_3=21$ for $C>2B$ and $\# N_3=20$ for $C=2B$.

{
\small
\renewcommand{\arraystretch}{1.6} 
\smallskip
\begin{longtable}{|c|c|c|c|}
\hline 
Node&Subset $R$&Condition&Topology of $U(R)$
\endhead
\hline
\circled{$a^4$} & $\{N-P, N-Q\}$ &---&  $2$-ball\\
\hline 
\circled{$b^4$} & $\{N-Q+P\}$ &---& $2$-ball\\
\hline 
\circled{$c^4$} & $(\nS_1\setminus\{P, Q, Q-P\})\cup\{Q-N,Q-N-P\}$ & (C2) & 
$2$ disjoint $2$-balls \\
\hline 
\circled{$d^4$} & $\{P,N-Q,N-Q+P\}$ &---& $2$-ball\\
\hline 
\circled{$e^4$} & $(\nS_1\setminus\{P,Q,Q-P\})\cup\{-P,-Q,P-Q\}$ &---& $2$-ball\\
\hline 
\circled{$f^4$} & $(\nS_1\setminus\{P,Q,Q-P\})\cup\{-Q, P-Q, P-N\}$ &---& $2$-ball\\
\hline 
\circled{$g^4$} & $ \{Q-P\}$ &---& $2$-ball\\
\hline 
\circled{$h^4$}& $\{Q,Q-P,N-P\}$ &---& $2$-ball\\
\hline 
\circled{$i^4$} & $\{N-P,N-Q,N-Q+P\}$ &---& $2$-ball\\
\hline 
\circled{$j^4$} & $(\nS_1\setminus\{P,Q,Q-P\})\cup\{-P,-Q\}$ &---& $2$-ball\\
\hline 
\circled{$k^4$} & $(\nS_1\setminus\{P,Q,Q-P\})\cup\{Q-N-P\}$ &---& $2$ disjoint $2$-balls\\
\hline 
\circled{$l^4$}&$(\nS_1\setminus\{P,Q,Q-P\})\cup(-\nS_1\setminus\{-N, P-Q, P-N\})$ &---& $2$-ball\\
\hline 
\circled{$m^4$} & $\{Q,Q-P,N-P,N-Q\}$ &---& $ 2$-ball\\
\hline 
\circled{$n^4$} & $\{Q,Q-P\}$ &(C2)& $2$-ball\\
\hline 
\circled{$o^4$}&$ \{N-Q,N-Q+P\}$ &(C2)&$ 2$-ball\\
\hline 
\circled{$p^4$}&$(\nS_1\setminus\{P,Q,Q-P\})\cup\{-Q,P-Q\}$ &(C2)& $2$-ball\\
\hline  
\caption{The set $N_4'$. \label{tab:N4} }
\end{longtable}
}

From the next step onwards we need to distinguish between the cases 
\begin{equation*}
\begin{split}
& \mathrm{(C1)} \qquad\qquad \qquad  A=1,\; B=2,\; C \ge 4;
\\
& \mathrm{(C2)} \qquad\qquad \qquad  A=1,\; B\geq 3,\; C\ge 2B.
\end{split}
\end{equation*}
With $N_4'$ as in Table~\ref{tab:N4} we gain $N_4=N_4'\cup N_3$. This entails that  $\# N_4=33$ for $A=1, B=2, C>4$ ($\#N_4=32 \text{ for } C=4$) and $\# N_4=37$ for $A=1, B\geq 3, C>2B$ ($\# N_4=36 \text{ for } C=2B$).

{
\small
\renewcommand{\arraystretch}{1.6} 
\smallskip
\begin{longtable}{|c|c|c|c|}
\hline 
Node&Subset $R$&Condition&Topology of $U(R)$
\endhead
\hline
\circled{$a^5$} & $\{N\}$ &---&  $2$-ball\\
\hline 
\circled{$b^5$} & $\{Q, N, N-P\}$ &---& $2$-ball\\
\hline 
\circled{$c^5$} & $\{N-Q,N-Q+P\}$ & (C1) & 
 $2$-ball \\
\hline 
\circled{$d^5$} & $\{Q,Q-P\}$ &(C1)& $2$-ball\\
\hline 
\circled{$e^5$} & $\{N,N-Q+P,-Q,P-Q\}$ &---& $2$-ball\\
\hline 
\circled{$f^5$} & $\{N, N-P, N-Q, N-Q+P, -Q, P-Q\}$ &(C1)& $2$-ball\\
\hline 
\circled{$g^5$} & $ (\nS_1\setminus\{P,Q,Q-P\}) \cup\{-Q\}$ &---& $2$-ball\\
\hline 
\circled{$h^5$}& $(\nS_1\setminus\{P,Q,Q-P\})\cup
\{-Q,-N,P-N\} $ &---& $2$-ball\\
\hline 
\circled{$i^5$} & $ (\nS_1\setminus\{P,Q,Q-P\})\cup\{-N\} $ &---& $2$ disjoint $2$-balls \\
\hline 
\circled{$j^5$} & $ \nS_1\setminus\{P,Q-P\}$ &---& $2$-ball\\
\hline 
\circled{$k^5$} & $\{N-Q\}$ &---&  $2$-ball\\
\hline
\circled{$l^5$} & $\{N,N-P\}$ &(C2)& $2$-ball\\
\hline 
\circled{$m^5$}&$(\nS_1\setminus\{P,Q,Q-P\})\cup
\{-N,P-N\}$ &(C2)&$2$ disjoint $2$-balls\\
\hline 
\circled{$m^3$} & $\{Q, N, Q-P, N-P\}$&$C=2B$&$2$-ball\\
\hline   
\caption{The set $N_5'$. \label{tab:N5} }
\end{longtable}
}

Now, $N_5=N_5'\cup N_4$, where $N_5'$ is given by Table~\ref{tab:N5}. Then $\# N_5=44$ for (C1) and $\# N_5=47$ for $(C2)$. As we indicated at the step that lead to $N_3$, at this stage the node \circled{$m^3$} is contained in $N_5$ in all the cases. Thus from now onwards we do not have to distinguish the cases $2B< C$ and $2B=C$. 

{
\small
\renewcommand{\arraystretch}{1.6} 
\smallskip
\begin{longtable}{|c|c|c|c|}
\hline 
Node&Subset $R$&Condition&Topology of $U(R)$
\endhead
\hline
\circled{$a^6$} & $\{N, N-P,-P,-Q\}$ &---&  $2$-ball\\
\hline 
\circled{$b^6$} & $\{N,N-P\}$ &(C1)& $2$-ball\\
\hline 
\circled{$c^6$} &$\{N-P\} $ &---& $2$-ball\\
\hline 
\circled{$d^6$} & $\{P\} $ &---& $2$-ball\\
\hline 
\circled{$e^6$} & $(-\nS_1\setminus\{P-Q\})\cup\{N,N-P\}$ &---& $2$-ball\\
\hline 
\circled{$f^6$} & $\{N,N-P,-N,Q-N-P\}$ &---& $2$ disjoint $2$-balls\\
\hline 
\circled{$g^6$} & $\{N, N-P, N-Q+P\}$ &---& $2$-ball\\
\hline 
\circled{$h^6$} & $\{P,Q\}$ &(C2)& $2$-ball\\
\hline 
\caption{The set $N_6'$. \label{tab:N6} }
\end{longtable}
}

We get $N_6=N_6'\cup N_5$ with $N_6'$ as in Table~\ref{tab:N6}.
Thus $\# N_6=51$ for $(C1)$ and $\# N_6=54$ for $(C2)$.

{
\small
\renewcommand{\arraystretch}{1.6} 
\smallskip
\begin{longtable}{|c|c|c|c|}
\hline 
Node&Subset $R$&Condition&Topology of $U(R)$
\endhead
\hline
\circled{$a^7$} & $ \{P,Q,Q-P\}$ &---&  $2$-ball\\
\hline 
\circled{$b^7$} & $ \{Q\}$ &---&  $2$-ball\\
\hline 
\circled{$c^7$} & $ \{P,Q\}$ &(C1)&  $2$-ball\\
\hline 
\circled{$d^7$} & $\{P, Q, N-Q,N-Q+P\}$ &---&  $2$-ball\\
\hline 
\caption{The set $N_7'$. \label{tab:N7} }
\end{longtable}
}

The next step of the iteration yields $N_7=N_7'\cup N_6$ with $N_7$ as in Table~\ref{tab:N7}. Thus $\# N_7=55$ for $A=1,B=2, C\geq 4 $ and $\# N_7=57$ for $A=1,B\geq 3,C\geq 2B$. 
We repeat the procedure once more and observe that $N_8=N_7$ for both conditions, so we have reached the end with $\# \nI=55$ for $A=1,B=2, C\geq 4 $ and $\# \nI=57$ for $A=1,B\geq 3,C\geq 2B$. 

It just remains to insert the edges of $\nI$ according to \eqref{eq:Nedge} in order to end up with the graphs depicted in Figure~\ref{Fig:B=2} and Figure~\ref{Fig:B>2} (and observing Remark~\ref{rem:graphremark}).
\end{proof}

The set $U'(R)=\bigcup_{\alpha\in R}O_\alpha$ with $R$ as in \circled{$d^4$} and \circled{$k^4$} is depicted on the left and right hand side of Figure~\ref{fig:ur}, respectively.

\begin{remark}\label{eq:circalpha}
For each $\alpha\in\nS$, $\overline{\{\alpha\}}$ is a node of $\nI$. In particular,
\circled{$d^6$}$=\overline{\{P\}}$,
\circled{$b^7$}$=\overline{\{Q\}}$,
\circled{$a^5$}$=\overline{\{N\}}$,
\circled{$g^4$}$=\overline{\{Q-P\}}$,
\circled{$c^6$}$=\overline{\{N-P\}}$,
\circled{$k^5$}$=\overline{\{N-Q\}}$, and
\circled{$b^4$}$=\overline{\{ N-Q+P\}}$. This will be of importance later.
\end{remark}

\begin{remark}
The number in the superscript of the labels of the nodes of $\nI$ indicates in which level of Figure~\ref{Fig:B=2} and Figure~\ref{Fig:B>2} a node occurs for the first time. For some nodes this happens at different levels in Figure~\ref{Fig:B=2} and Figure~\ref{Fig:B>2}. In these cases we gave different names to this node in the two graphs. So we have \circled{$n^4$}=\circled{$d^5$}, \circled{$o^4$}=\circled{$c^5$}, \circled{$l^5$}=\circled{$b^6$}, \circled{$h^6$}=\circled{$c^7$}. This fact is of no relevance in the sequel. Only the classes $\overline{R}$ corresponding to \circled{$c^4$} and \circled{$m^5$} are in $\nI$ under condition $(C2)$ but not under $(C1)$.
We just did it that way because it makes it easier to locate the first occurrence of a given node in the figures.
\end{remark}

\begin{lem}\label{lem:auxt1t2}
Let $T$ be an $ABC$-tile with $14$ neighbors and assume that  $A=1$.
Let $t_1,t_2\in \nC$ be essentially disjoint with $\level{t_1} \le \level{t_2}$.
Let $t_2' \in \nC$ with 
\begin{equation}\label{eq:t'tb}
\level{\sub_2'}=\level{\sub_2}+1 \quad\hbox{and}\quad \sub_2'\subset \sub_2. 
\end{equation}
Assume that the type of $\sub_1 \cap \sub_2$ is $\overline{R}\in \nI$. 
Then the type $\overline{R'}$ of $\sub_1 \cap \sub_2'$ is either $\overline{\emptyset}$ or 
$\overline{R} \to \overline{R'} \in \nI$.
\end{lem}

\begin{proof}
Assume first that $\sub_1=\sub_\infty$. Set $\level{\sub_2}=i$. Then there is $d=d_{i}+\cdots+M^{i-1}d_{1} \in \D_{i}$ and $d_{i+1}\in \D$ such that $\sub_2=M^{-i}(T+d)$ and $\sub_2'=M^{-i-1}(T+d_{i+1}+Md)$. Now \eqref{cell-i} and \eqref{eq:Rgginf} yield
\[
\begin{split}
R&=\{\alpha_i \;:\, \hbox{there is $\alpha\in \nS$ with } \alpha\xrightarrow{d_1}\alpha_1\xrightarrow{d_2}\cdots\xrightarrow{d_i}\alpha_i\in G(\nS) \},
\\
R'&=\{\alpha_{i+1} \;:\, \hbox{there is $\alpha\in \nS$ with } \alpha\xrightarrow{d_1}\alpha_1\xrightarrow{d_2}\cdots\xrightarrow{d_i}\alpha_i\xrightarrow{d_{i+1}}\alpha_{i+1}\in G(\nS) \}.
\end{split}
\]
%
Thus \eqref{eq:ndrDef} yields $R' =n_{d_{i+1}}(R)$, and by \eqref{eq:Nrec} and \eqref{eq:Nedge} either $\overline{R'}=\overline{\emptyset}$ or $\overline{R'}$ satisfies $\overline{R} \to \overline{R'} \in \nI$.

If $t_1\not=t_\infty$ the result follows analogously by using \eqref{eq:gg12} and \eqref{eq:gg13} instead of \eqref{cell-i} and~\eqref{eq:Rgginf}.
\end{proof}

\begin{prop}\label{prop:capGraph}
Let $T$ be an $ABC$-tile with $14$ neighbors and assume that  $A=1$.
If $\sub_1,\sub_2 \in \nC$ are essentially disjoint then the type $\overline{R}$ of $\sub_1\cap\sub_2$ is either $\overline{\emptyset}$ or $\overline{R}\in\nI$.
In particular, the following assertions hold.
\begin{itemize}
\item[(1)]
Let $\sub_1=\sub_\infty$ and $\sub_2 \in \nC\setminus\{\sub_\infty\}$ with $\level{\sub_2}=i$. Then either $\sub_\infty \cap \sub_2 =\emptyset$ or there is a walk
$
\overline{\nS} \rightarrow \overline{R_1} \rightarrow \overline{R_2} 
\rightarrow \cdots \rightarrow \overline{R_i} 
$
of length $i$ in $\nI$ such that $\sub_\infty \cap \sub_2$ is of type $\overline{R_i} $ and, hence, $\sub_\infty \cap \sub_2 \simeq U(R_i)$.

\item[(2)] Let $\sub_1,\sub_2 \in \nC\setminus\{\sub_\infty\}$ be essentially disjoint with $\level{\sub_1}\le\level{\sub_2}$ and let $i=\level{\sub_2}-\level{\sub_1}$. Then either $\sub_1 \cap \sub_2 =\emptyset$ or there is $\alpha\in \nS$ and a walk
$
\overline{\{\alpha\}} \rightarrow \overline{R_1} \rightarrow \overline{R_2} 
\rightarrow \cdots \rightarrow \overline{R_i} 
$
of length $i$ in $\nI$ such that $\sub_1 \cap \sub_2$ is of type $\overline{R_i}$ hence, $\sub_1 \cap \sub_2\simeq U(R_i)$.
\end{itemize}
\end{prop}

\begin{proof}
We prove $(1)$. The proof is done by induction on $i=\level{\sub_2}$. If $i=0$ then $\sub_2=T$ and, hence, the type of $\sub_\infty \cap \sub_2$ is $\overline{R}=\overline{\nS} \in \nI$. 

For the induction hypothesis assume that the result holds for all $\sub_2\in\nC$ with $0\le \level{\sub_2} \le i-1$. 

For the induction step let $\sub_2'\in\nC$ with $\level{\sub_2'}=i$ be given and assume that $\sub_2$  satisfies \eqref{eq:t'tb}. Then by the induction hypothesis the type $\overline{R_{i-1}}$ of $\sub_\infty \cap \sub_2$ is either $\overline{\emptyset}$ or there is a walk $\overline{\nS} \rightarrow \overline{R_1} \rightarrow \overline{R_2} 
\rightarrow \cdots \rightarrow \overline{R_{i-1}}$ of length $i-1$
in $\nI$. If $\overline{R_{i-1}}=\overline{\emptyset}$ then \eqref{eq:t'tb} implies that $\sub_\infty \cap \sub_2'=\emptyset$, hence, its type is $\overline{\emptyset}$ as well, and we are done. If  $\overline{R_{i-1}} \in \nI$ then by Lemma~\ref{lem:auxt1t2} the type $\overline{R_i}$ of $\sub_\infty \cap \sub_2'$ is either $\overline{\emptyset}$ or satisfies $\overline{R_{i-1}} \to \overline{R_i} \in \nI$. In the latter case there is a walk $\overline{\nS} \rightarrow \overline{R_1} \rightarrow \overline{R_2} 
\rightarrow \cdots \rightarrow \overline{R_{i}}$ of length $i$
in $\nI$. This finishes the induction step.

The case $\sub_1\not=\sub_\infty$ follows analogously by induction on $\level{\sub_2}-\level{\sub_1}$. Just note that, if $\level{\sub_1}=\level{\sub_2}$, then $\sub_1\cap \sub_2$ is either empty or has type $\overline{\{\alpha\}} \in \nI$ for some $\alpha\in \nS$ (observe Remark~\ref{eq:circalpha}).

The  fact that $\sub_1 \cap \sub_2 \simeq U(R)$ if it has type $U(R)$ is already contained in Lemma~\ref{lem:typetopo}.
\end{proof}

\section{Proofs of the main results}\label{sec:4}

This section is devoted to the proofs of our main results. In Section~\ref{sec:par} we recall the definition of partitionings in the sense of Bing~\cite{Bing51} and give some results on partitionings that will be needed in the sequel. In Section~\ref{sec:seqpart} we define sequences of partitionings that are suitable for our purposes. In Section~\ref{sec:order} we make sure that in these sequences each atom is subdivided in a way that certain connectivity properties are maintained. Finally, Sections~\ref{sec:proof1} and~\ref{sec:proof2} contain the proofs of Theorem~\ref{thm:ball} and Theorem~\ref{thm:CW}, respectively.

\subsection{Partionings}\label{sec:par}
In this section we give the definitions and results of  Bing's theory of partitionings~\cite{Bing51} that will be relevant for the proof of Theorem~\ref{thm:ball}. We start with some terminology.

\begin{definition}[Partitioning]
Let $X$ be a metric space. A {\em partitioning} of $X$ is a collection of mutually disjoint open sets (so-called {\it atoms}) whose union is dense in $X$. A partitioning is called \emph{regular} if each of its atoms is the interior of its closure. Let $G$ and $G'$ be two partitionings of $X$. $G'$ is a {\it refinement} of $G$ if for each $g'\in G'$ there exists $g\in G$ with $g'\subseteq g$.
A sequence $(G_i)_{i\ge 1}$ of partitionings is called a {\em decreasing sequence of partitionings} if $G_{i+1}$ is a refinement of $G_i$ and the maximum of the diameters of the atoms of $G_i$ tends to $0$ as $i$ tends to infinity.
\end{definition}

\begin{definition}[Equivalent sequences of partitionings]\label{Def:Homeo}
Let $X_1$ and $X_2$ be two metric spaces. Let $(G_{ij})_{j\ge 1}$ be a sequence of partitionings of $X_i$ for each $i\in \{1,2\}$. We say that $(G_{1j})$ and  $(G_{2j})$ are {\em equivalent partitionings}, if for each $j\ge 1$ there exists a 1 to 1 correspondence between the atoms of $G_{1j}$ and $G_{2j}$ such that
\begin{itemize}
\item[(1)] two atoms of $G_{1j}$ have a boundary point in common if and only if the corresponding atoms of $G_{2j}$ have a boundary point in common.
\item[(2)] corresponding atoms of $G_{1,j+1}$ and $G_{2,j+1}$ are subsets of corresponding atoms of $G_{1j}$ and~$G_{2j}$.
\end{itemize}
If $(G_{1j})$ and  $(G_{2j})$ are equivalent we write $(G_{1j})\sim (G_{2j})$.

We say that two finite sequences $(G_{ij})_{j=1}^n$ of partitionings of $X_i$ ($i\in \{1,2\}$) are {\em equivalent} if  for each $j\in\{1,\ldots, n\}$ there exists a 1 to 1 correspondence between the atoms of $G_{1j}$ and $G_{2j}$ such that (1) holds for $1\le j \le n$ and (2) holds for $1\le j <n$.
\end{definition}

\begin{remark}\label{rem:simeq}
It is easy to check that the relation ``$\sim$'' is an equivalence relation.
\end{remark}

The following lemma, which can be easily proved, is just a reformulation of \cite[Theorem~6]{Bing51}.

\begin{lem}\label{lem:homeo}
Two Peano continua $X_1$ and $X_2$ are homeomorphic if and only if for each $i\in\{1,2\}$ there exists a decreasing sequence of partitionings $(G_{ij})_{j\ge1}$ for $X_i$ such that $(G_{1j})_{j\ge1}\sim (G_{2j})_{j\ge1}$.
\end{lem}

This lemma will be used in the proof of Theorem~\ref{thm:ball}. Indeed, we will construct a decreasing sequence of partitionings for the self-affine tile $T$ (which is a Peano continuum by Lemma~\ref{lem:HataPeano}) that is equivalent to a decreasing sequence of partitionings of $\mathbb{D}^3$. In the course of our proof we will use the following two results from \cite{Bing51}. The first one is about the extension of homeomorphisms. Recall that a $2$-sphere $C$ in $\R$ is {\em tame} if there is a homeomorphism from $\R^3$ to $\R^3$ that maps $C$ to the unit sphere $\mathbb{S}^2$ in $\R^3$.

\begin{prop}[{see \cite[Theorem~3]{Bing51}}]\label{prop:bing3}
Let $S$ be a Peano continuum and $S_2\subset S$ a $2$-sphere. Let $C\subset \R^3$ be a tame $2$-sphere and $F:S_2\to C$ a homeomorphism. Assume that $G$ is a regular partitioning of $S$ satisfying the following conditions.
\begin{itemize}
\item[(1)] If $g \in G$ then $\partial g \cong \mathbb{S}^2$.
\item[(2)] If $g_1,g_2\in G$ are distinct then $\partial g_1 \cap \partial g_2$ is either empty or  a finite union of mutually disjoint $2$-balls.
\item[(3)] If $g_1,g_2,g_3\in G$ are mutually distinct then $\partial g_1 \cap \partial g_2 \cap \partial g_3$ is either empty or a finite union of arcs.
\item[(4)] There exist $g_1,\ldots, g_n \in G$ such that $S_2 =\partial(\overline{g_1\cup\dots\cup g_n})$ and  such that the intersection $\partial g_j \cap (S_2\cup \partial g_1\cup\dots\cup \partial g_{j-1})$ is connected for each $j\in \{1,\ldots, n\}$. 
\end{itemize}
Then there is a partitioning $\{h_0,h_1,\ldots,h_n\}$ of ${\R}^3$ and a homeomorphism $F':\partial_S(g_1\cup\dots\cup g_n) \to \partial_{\mathbb{R}^3}(h_1\cup\dots\cup h_n)$ such that $h_0$ is the exterior of $C$ and $\partial h_i$ is a tame $2$-sphere, $F=F'$ on $S_2$, and $F'(\partial g_i) = \partial h_i$ ($1\le i \le n$). 
\end{prop}

The next result will be used in the proof of Theorem~\ref{thm:ball} in the context of decreasing sequences of partitionings.

\begin{prop}[{see \cite[Theorem~5]{Bing51}}]\label{prop:bing5}
Let $C\subset \mathbb{R}^3$ be a tame $2$-sphere and $(G_i)_{i\ge 1}$ a sequence of partitionings of ${\R}^3$ satisfying the following conditions for each $i\ge 1$.
\begin{itemize}
\item[(1)] If $g \in G_i$ then $\partial g \cong \mathbb{S}^2$ is tamely embedded in $\R^3$.
\item[(2)] For each $g\in G_i$ with $\overline{g}\cap C\not=\emptyset$ the set $\partial g \cap C$ is connected and does not separate $\partial g$.
\item[(3)] $G_{i+1}$ is a refinement of $G_i$.
\item[(4)] One atom $g_0\in G_i$ is the exterior of $C$.
\item[(5)] For each $\varepsilon > 0$ and each $i\in\N$ there is $n=n(i,\varepsilon)\ge 1$ such that $\overline{g'} \cap \bigcup_{g\in G_i} \partial g$ has diameter less than $\varepsilon$ for each $g'\in G_n\setminus\{g_0\}$.
\end{itemize}  
Then for each $\delta >0$ there is $m\ge 1$ and a homeomorphism $F:{\R}^3\to{\R}^3$ such that $F$ leaves each point of $g_0$ invariant and ${\rm diam}(F(g))<\delta$ for each $g\in G_m\setminus\{g_0\}$. 
\end{prop}

\subsection{Sequences of partitionings}\label{sec:seqpart}

Let $T=T(M,\D)$ be an $ABC$-tile. In Section~\ref{sec:intersec} it was convenient to work with closed sets (the subtiles of $T$). When it comes to partitionings, open sets are required. Therefore, in the sequel we will mainly work  with the interiors of subtiles. Moreover, we often use the one point compactification $\mathbb{S}^3 = \mathbb{R}^3\cup\{\infty\}$ as ambient space because $\mathbb{S}^3$ is a Peano continuum. The following lemma provides a first sequence of partitionings defined in terms of interiors of subtiles. We will frequently use the notation $g_{\infty}=\mathbb{S}^3\setminus T$ in the sequel. Note that $g_\infty = \sub_\infty^\circ \cup \{\infty\}$ and, hence, $\partial_{\mathbb{S}^3}g_\infty=\partial_{\mathbb{R}^3}\sub_\infty$.

\begin{lem}\label{lem:PiPart}
Let $T$ be an $ABC$-tile. Let $g_{\infty}=\mathbb{S}^3\setminus T$. Then for each $i\ge 0$ the collection
\begin{equation}\label{eq:partitioning}
\nP_i =  \left\{ M^{-i}(T+z)^\circ  \;:\; z \in  \D_{i}  \right\}
 \cup \{g_{\infty}\}
 \end{equation}
is a regular partitioning of $\mathbb{S}^3$. Moreover $\nP_i\setminus\{g_\infty\}$ is a regular partitioning of $T$.
 \end{lem}
 
\begin{proof}
Let $i\ge 0$. By \eqref{eq:iset} we have $T= \bigcup_{d\in \D_i}M^{-i}(T+d)$. Since each subtile $M^{-i}(T+d)$, $d\in \D_i$, is the closure of its interior by Lemma~\ref{lem:LW1}, we see form \eqref{eq:intersectDi} that $\nP_i\setminus\{g_\infty\}$ is a regular partitioning of $T$. Thus $\nP_i$ is a regular partitioning of  $\mathbb{S}^3$.
 \end{proof}

We use the definition of {\it level} from \eqref{eq:leveldef} also for the elements of $\nP_i$, $i\ge 0$. Indeed, we set 
\[
\level{g} = \level{\overline{g} \setminus\{\infty\}} 
\]
for $g\in \bigcup_{i\ge 0} \nP_i$. As usual, for a subset $Y\subset  \bigcup_{i\ge 0} \nP_i$ we will write $\level{Y}=\{\level{g}\,:\, g\in Y\}$.

In view of  \eqref{eq:inteq} an intersection $\partial g_1 \cap \partial g_2$ for disjoint atoms $g_1,g_2 \in \bigcup_{i\ge 0} \nP_i $ is equal to the intersection $\overline{g_1} \cap \overline{g_2}$ of the corresponding elements $\overline{g_1},\overline{g_2}\in\nC$. 

We continue with topological properties of intersections of boundaries of the atoms of $\nP_i$. 

\begin{lem}\label{lem:division}
Let $T$ be an $ABC$-tile with $14$ neighbors and assume that  $A=1$. Let  $i\ge 2$. For any $g\in \nP_i\setminus\{g_\infty\}$, the intersection $\partial g_\infty \cap \partial g$ is either empty, or a union of at most $2$ disjoint $2$-balls.
\end{lem}
\begin{proof}
Suppose that $\partial g_\infty \cap \partial g \neq\emptyset$. Then by  Proposition~\ref{prop:capGraph}~(1) the intersection $\partial g_\infty \cap \partial g$ is homeomorphic to $U(R)$, where $R\subseteq \nS$ is a representative of a node $\overline{R}$ of $\nI$. We can now read off Figures~\ref{Fig:B=2} and~\ref{Fig:B>2} that in this case $U(R)$ is either a union of at most $2$ disjoint $2$-balls, or homeomorphic to $\mathbb{S}^1\times[0,1]$ (a ``ribbon''), or homeomorphic to $\mathbb{S}^2$. We need to exclude the last two cases. Suppose that $\partial g_\infty \cap \partial g$ is homeomorphic to $\mathbb{S}^1\times[0,1]$ or $\mathbb{S}^2$. Then, because $g\in \nP_i\setminus\{g_\infty\}$ we have $\level{g}=i$ and, according to Proposition~\ref{prop:capGraph}~(1)  there is a walk $\overline{\nS} \rightarrow \overline{R_1}\rightarrow \cdots \rightarrow \overline{R_i}$ in $\nI$ with $U(R_i)$ being homeomorphic to $\mathbb{S}^1\times[0,1]$ or $\mathbb{S}^2$. We know that $\overline{\nS}$ and \circled{$b^1$} are the only nodes of $\nI$ homeomorphic to $\mathbb{S}^1\times[0,1]$ or $\mathbb{S}^2$. However, as we see from Figure~\ref{Fig:B=2} and Figure~\ref{Fig:B>2}, there is no walk of length $i\ge 2$ in $\nI$ ending at \circled{$b^1$} or $\nS$. Thus $\partial g_\infty \cap \partial g$ can neither be homeomorphic to $\mathbb{S}^1\times[0,1]$ nor to~$\mathbb{S}^2$.
\end{proof}

In view of Lemma~\ref{lem:division} we can subdivide the atoms of $\nP_i\setminus\{g_\infty\}$, $i\ge 2$, according to the way they intersect $\partial g_\infty = \partial T$. In particular, for $i\ge 2$ set 
\begin{equation}\label{eq:p1prime}
\begin{aligned}
\nP_{i1}&=\{g\in \nP_i\setminus\{g_\infty\} \;:\; \partial g\cap \partial T=\emptyset \},\\
\nP_{i2}&=\{g\in \nP_i\setminus\{g_\infty\} \;:\; \partial g\cap \partial T \text{ is  a single $2$-ball} \},\\
\nP_{i3}&=\{g\in \nP_i\setminus\{g_\infty\} \;:\; \partial g\cap \partial T \text{ is the union of $2$ disjoint $2$-balls} \}.
\end{aligned}
\end{equation}
 Then we have $\nP_i=\nP_{i1}\cup \nP_{i2}\cup \nP_{i3}\cup \{g_{\infty}\}$. We need partitionings whose atoms have intersections with $\partial T$ that are either empty or a $2$-ball. To achieve this we further subdivide the atoms of $\nP_{i3}$ and put, again for $i\ge 2$,
 \begin{equation*}
 \begin{aligned}
\nQ_{i1}=&\nP_{i1},\\
\nQ_{i2}=&\nP_{i2},\\
\nQ_{i3}=&\{g\in \nP_{i+1}\;:\; g\subset g' \text{ for } g'\in \nP_{i3} \}.
\end{aligned}
\end{equation*}
Let $(\nQ'_i)_{i\ge 1}$ be given by
\begin{equation}\label{new_par}
\begin{split}
\nQ'_1&=\{T^\circ\},\\
\nQ'_i&=\nQ_{i1}\cup \nQ_{i2}\cup \nQ_{i3} \qquad(i\ge 2),
\end{split}
\end{equation}
and set $\nQ_i= \nQ'_i\cup\{g_\infty\}$ for $i\ge 1$. From this definition we immediately get
\begin{equation}\label{eq:QiLevel}
\level{\nQ'_1}=0 \quad \hbox{and} \quad \level{\nQ'_i} = \{i,i+1\} \hbox{ for } i\ge 2.
\end{equation}

\begin{lem}\label{lem:regPprim}
Let $T$ be an $ABC$-tile with $14$ neighbors and assume that  $A=1$. The sequence  $(\nQ_i')_{i\ge 1}$ given by \eqref{new_par} is a decreasing sequence of regular partitionings of $T$.
\end{lem}

\begin{proof}
For $i=1$ the collection $\nQ_1'$ is clearly a regular partitioning of $T$. Let now $i\ge 2$.
Let $g\in \nP_i\setminus\{g_\infty\}$ be given. Then $g=M^{-i}(T+d)^\circ$ for some $d\in \D_i$. We claim that $X_g=\{h\in \nQ_i' \,:\, h\cap g \not=\emptyset\}$ is a regular partitioning of the Peano continuum $\overline{g}$. If $g\in \nP_{i1} \cup \nP_{i2}$ then
$X_g=\{g\}$ and the claim is trivial. If $g\in \nP_{i3}$ then  $X_g = \{M^{-i-1}(T+d_i+Md)^\circ\,:\, d_i\in \D\}$ and the claim follows from Lemma~\ref{lem:LW1} and \eqref{eq:intersectDi} because $\overline{g}=\bigcup_{d_i\in \D}M^{-i-1}(T+d_i+Md)$ by the set equation \eqref{eq:setequation}.  This proves the claim in all cases. Since
\[
\nQ_i'= \bigcup_{g\in \nP_i \setminus \{g_\infty\}} X_g,
\]
$X_g$ is a regular partitioning of $\overline{g}$ for each $g\in \nP_i\setminus \{g_\infty\}$, and $\nP_i\setminus \{g_\infty\}$ is a regular partitioning of $T$ by Lemma~\ref{lem:PiPart}, we conclude that $\nQ_i'$ is a regular partitioning of $T$ as well.

Because $M^{-1}$ is a uniform contraction, $\max\{\mathrm{diam}\,g \,:\, g\in \nP_i\setminus\{g_\infty\}\} = \mathrm{diam}\, M^{-i}T \to 0$ for $i\to \infty$.
The fact that $(\nQ_i')_{i\ge 1}$ is decreasing now follows because by \eqref{eq:QiLevel}, $\nQ_{2}'$ is a refinement of $\nQ_{1}'$ and, for each $i\ge 2$, $\nQ'_{i+1}$ is a refinement of $\nP_{i+1}\setminus\{g_\infty\}$ and $\nP_{i+1}\setminus\{g_\infty\}$ is a refinement of~$\nQ'_{i}$.
\end{proof}

Let $g\in \bigcup_{i\ge 0}\nP_i \setminus \{g_\infty\}$ be given. Then there is $k\ge 0$ and $d\in \D_k$ such that $g=M^{-k}(T^\circ+d)$. In this case we associate with $g$ the mapping $[g]:\R^3 \to \R^3$, $x \mapsto M^{-k}(x+d)$. If $H$ is a collection of sets, then we set $[g](H) = \{[g](h) \,:\, h\in H \}$. Clearly, if $H$ is a partitioning of a Peano continuum $X$, then $[g](H)$ is a partitioning of $[g](X)$. 
We need the following generalization of $(\nQ'_i)$ and $(\nQ_i)$. Let $\mathbf{n}=(n_j)_{j\ge 1}$ be a sequence with $n_j\in \N\cup\{\infty\}$ satisfying $n_1\ge 3$ and $n_{j+1}-n_j \ge 3$ (we allow that $\mathbf{n}$ can become ultimately $\infty$, {\it i.e.}, for each $n\in\N$ we define $n< \infty$ and $\infty + n \le  \infty$). We define the sequence of partitionings $(\nQ'_i(\mathbf{n}))_{i\ge 1}$ by \\[-2mm]
\begin{equation}\label{eq:nPn}
\begin{array}{rcll}
\displaystyle \nQ'_i(\mathbf{n}) &=& \displaystyle\nQ'_i & (1\le i < n_1), \\[2mm]
\displaystyle\nQ'_{n_j}(\mathbf{n}) &=& \displaystyle \nQ'_{n_j-1}(\mathbf{n}) &(j\ge 1), \\[2mm]
\displaystyle\nQ'_i(\mathbf{n}) &=& \displaystyle \bigcup_{g\in \nQ'_{n_j}(\mathbf{n})} [g](\nQ'_{i-\level{g}}) \quad&  (n_j< i < n_{j+1},\; j \ge 1). 
\end{array}
\end{equation}
Moreover, set $\nQ_i(\mathbf{n}) = \nQ'_i(\mathbf{n}) \cup\{g_\infty\}$ for $i\ge 1$. Note that $(\nQ_i)_{i\ge 1}=(\nQ_i(\mathbf{n}))_{i \ge 1}$ if $\mathbf{n}=(n_j)_{j\ge 1}$ satisfies $n_j=\infty$ for each $j\ge 1$. 

\begin{remark}
The definition of $(\nQ'_i(\mathbf{n}))$ is a bit technical. Its main feature is a repetitivity property. After $n_j$ steps each atom of $\nQ'_{n_j}(\mathbf{n})$ is subdivided in the same way as $T$ itself ({\it i.e.}\ by using partitionings equivalent to $\nQ_i'$) for $n_{j+1}-n_{j}-1$ steps. Sloppily speaking, in $\nQ'_{n_j}(\mathbf{n})$ each atom is subdivided by the ``nice'' subdivision equivalent to~$\nQ_i'$ for some time. This repetitivity, which is not present in $(\nQ_i')$, will be of importance later.
\end{remark}

The next result contains basic properties of the sequence of partitionings $(\nQ_i(\mathbf{n}))_{i\ge 1}$.

\begin{lem}\label{lem:nBasic}
Let $T$ be an $ABC$-tile with $14$ neighbors and assume that  $A=1$. Let $\mathbf{n}=(n_j)_{j\ge 1}$ be a sequence with $n_j\in \N\cup\{\infty\}$ satisfying $n_1\ge 3$ and  $n_{j+1}-n_j\ge 3$ and let $(\nQ'_i(\mathbf{n}))_{i\ge 1}$ be as in~\eqref{eq:nPn}.
Then 
\begin{itemize}
\item[(i)] $g\in \nQ'_i(\mathbf{n})$ implies $\level{g}\in\{i-1,i,i+1\}\setminus\{1\}$ ($i\ge 1$).
\item[(ii)] $\nQ'_i(\mathbf{n})$ is a regular partitioning of $T$ ($i\ge 1$).
\item[(iii)] $(\nQ'_i(\mathbf{n}))_{i\ge 1}$ is a decreasing sequence of partitionings of $T$.
\end{itemize}
\end{lem}

\begin{proof}
To prove (i) we first prove the following more detailed result (set $n_0=0$ for convenience).
\begin{itemize}
\item[(a)] If $n_{j-1}+1<i<n_{j}$ then $g\in \nQ'_i(\mathbf{n})$ implies $\level{g}\in\{i,i+1\}$ ($j\ge 1$).
\item[(b)] If $i=n_{j}$ then $g\in \nQ'_i(\mathbf{n})$ implies $\level{g}\in\{i-1,i\}$ ($j\ge 1$).
\item[(c)] If $i=n_{j}+1$ then $g\in \nQ'_i(\mathbf{n})$ implies $\level{g}\in\{i-1,i,i+1\}$ ($j\ge 0$).
\end{itemize}
This is proved by induction on $j$. For $1\le i\le n_1$ we have $\nQ'_i(\mathbf{n})=\nQ'_{\min\{i,n_1-1\}}$ and the result follows from \eqref{eq:QiLevel}. Suppose that the result holds for $i\le n_{j}$. If $n_j < i < n_{j+1}$ then
\begin{equation}\label{Qprimeexpl}
\nQ'_i(\mathbf{n}) =  \bigcup_{g\in \nQ'_{n_j}(\mathbf{n})} [g](\nQ'_{i-\level{g}})
\end{equation}
Let $g'\in \nQ'_i(\mathbf{n})$.
Assume first that $i=n_j+1$. Then, because $\level{\nQ'_{n_j}(\mathbf{n})} = \{n_j-1,n_j\}$  this implies that either $g' \in [g](\nQ'_2)$ for some $g$ with $\level{g} = n_j-1$, hence, by \eqref{eq:QiLevel}, $\level{g'}\in\{n_j+1,n_j+2\}$, or $g' \in [g](\nQ'_1)$ for some $g$ with $\level{g} = n_j$, hence, by \eqref{eq:QiLevel}, $\level{g'}=n_j$ which is~(c). 
If $n_{j}+1<i<n_{j+1}$ then, because $\level{\nQ'_{n_j}(\mathbf{n})} = \{n_j-1,n_j\}$  this implies that 
either $g' \in [g](\nQ'_{i-n_j+1})$ for some $g\in \nQ'_{n_j}(\mathbf{n})$ with $\level{g} = n_j-1$, hence, by \eqref{eq:QiLevel}, $\level{g'}\in\{i,i+1\}$, or $g' \in [g](\nQ'_{i-n_j})$ for some $g\in\nQ'_{n_j}(\mathbf{n})$ with $\level{g} = n_j$, hence, by \eqref{eq:QiLevel}, $\level{g'}\in\{i,i+1\}$ which is~(a). 
If $i=n_{j+1}$, (b)~follows immediately from~(a). This finishes the induction proof of (a), (b), and~(c).

Finally let $g\in\nQ'_i(\mathbf{n})$ for some $i\ge 1$.
Now (i) follows from (a), (b), and~(c) because for $i=1$ we have $\level{g}=0$, for $i= 2$ we have $\level{g}\in\{2,3\}$ (since $n_1 \ge 3$), and for $i\ge 3$ we have $\level{g} \ge 2$. Thus $\level{g}=1$ cannot occur for any $g\in \nQ'_i(\mathbf{n})$, $i\ge 1$.

To prove (ii) we use again induction on $j$. For $1\le i \le n_1$ the collection $\nQ_1'(\mathbf{n})$ is a regular partitioning of $T$ by Lemma~\ref {lem:regPprim}. Let now $j\ge 2$. Since $\nQ'_{n_{j+1}}(\mathbf{n}) = \nQ'_{n_{j+1}-1}(\mathbf{n})$ we may assume that $n_j < i < n_{j+1}$. In this case $[g](\nQ'_{i-\level{g}})$ is a regular partitioning of $\overline{g}$ for each $g\in \nQ'_{n_j}(\mathbf{n})$ by Lemma~\ref{lem:regPprim}, and $\nQ'_{n_j}(\mathbf{n})$ is a regular partitioning of $T$ by the induction hypothesis. Thus by \eqref{Qprimeexpl} also $\nQ'_i(\mathbf{n})$ is a regular partitioning of $T$, and the induction is finished. This proves~(ii).

For (iii) we first show that $\nQ'_{i+1}(\mathbf{n})$ is a refinement of $\nQ'_{i}(\mathbf{n})$. 
For $n_j +1 < i < n_{j+1}-1$ this follows from (a). For $i=n_{j}-1$ it follows because $\nQ'_{n_j-1}(\mathbf{n}) =\nQ'_{n_j}(\mathbf{n})$, for $i=n_j$ it follows form (b) and (c), and for $i=n_{j}+1$ it follows from (c) and (a). The fact that $(\nQ'_i(\mathbf{n}))_{i\ge1}$ is decreasing follows from (i) because $M^{-1}$ a uniform contraction. 
\end{proof}

The following result contains some topological properties of $(\nQ_i(\mathbf{n}))_{i\ge 1}$ that are related to some of the conditions of Propositions~\ref{prop:bing3} and~\ref{prop:bing5}.

\begin{prop}\label{lem:topologyP}
Let $T$ be an $ABC$-tile with $14$ neighbors and assume that  $A=1$. Let $\mathbf{n}=(n_j)_{j\ge 1}$ be a sequence with $n_j\in \N\cup\{\infty\}$ satisfying $n_1\ge 3$ and  $n_{j+1}-n_j \ge 3$.
Then the following conditions hold for $i\ge 2$.
\begin{itemize}
\item[(1)]  For each $g\in \nQ_i(\mathbf{n})$ we have  $\partial g \cong \mathbb{S}^2$.
\item[(2)] If $g_1,g_2\in \nQ_i(\mathbf{n})$ are distinct then $\partial g_1 \cap \partial g_2 $ is either empty or a union of at most $2$ disjoint $2$-balls. 
\item[(3)] If  $g_1,g_2 \in \nQ_i$ are distinct then $\partial g_1 \cap \partial g_2$ is either empty or a single $2$-ball.\footnote{We need this item only for $g_1=g_\infty$ but give the more general case for the sake of completeness.}
\item[(4)] If $g_1,g_2,g_3\in \nQ_i(\mathbf{n})$ are distinct then $\partial g_1 \cap \partial g_2 \cap \partial g_3 $ is either empty or a finite union of arcs.
\end{itemize}
\end{prop}

\begin{proof}
Throughout the proof we assume that $i\ge 2$.

Each $g\in\nQ_i(\mathbf{n})$ either satisfies $g=\mathbb{S}^3\setminus T$ or $g=M^{-j}(T+z)^{\circ}$ for some $j \in \{i-1,i,i+1\}\setminus\{1\}$ and some $z\in \Z^3$ by Lemma~\ref{lem:nBasic}~(i). In any case $\partial g $ is homeomorphic to $\partial T$ and, hence, item~(1) follows from \cite[Theorem~1.1~(1)]{TZ:19}.  

If $g_1,g_2 \in \nQ_i'(\mathbf{n})$ then, after possibly exchanging $g_1$ and $g_2$, Lemma~\ref{lem:nBasic}~(i) implies that there are $k,l \in\{i-1,i,i+1\}\setminus\{1\}$ such that $k\le l$, $\level{g_1}=k$ and $\level{g_2}=l$. 
Assume that $\partial g_1 \cap \partial g_2\not=\emptyset$.
Thus Proposition~\ref{prop:capGraph}~(2) implies that the intersection $\partial g_1 \cap \partial g_2$ is homeomorphic to $U(R)$ for a node $\overline{R}$ of $\nI$ which can be reached from one of the nodes $\overline{\{\alpha\}}$, $\alpha\in\nS$, by a walk of length zero, one, or two. 
Since we see from Figures~\ref{Fig:B=2} and~\ref{Fig:B>2} (recall Remark~\ref{eq:circalpha}) that all these nodes correspond to a 
$2$-ball, item~(2)  follows for this case. It remains to show item (2) for the case $g_1=g_\infty$.  Because $i\ge 2$ we have $\level{g_2}\ge 2$ by Lemma~\ref{lem:nBasic}~(i) and (2) follows from Lemma~\ref{lem:division}

To prove~(3) let first $g_1,g_2 \in \nQ_i'$ then, after possibly exchanging $g_1$ and $g_2$, by the definition of  $\nQ_i'$ there are $k,l \in\{i,i+1\}$ such that $k\le l$, $\level{g_1}=k$ and $\level{g_2}=l$. Assume that $\partial g_1 \cap \partial g_2\not=\emptyset$. Then Proposition~\ref{prop:capGraph}~(2) implies that the intersection $\partial g_1 \cap \partial g_2$ is homeomorphic to $U(R)$ for a node $\overline{R}$ of $\nI$ which can be reached from one of the nodes $\overline{\{\alpha\}}$, $\alpha\in\nS$, by a walk of length zero or one. Since we see from Figures~\ref{Fig:B=2} and~\ref{Fig:B>2} (recall again Remark~\ref{eq:circalpha}) that all these nodes correspond to a $2$-ball, item~(2)  follows for this case. Let now $g_1= g_\infty$ and assume that $\partial g_\infty \cap \partial g_2\not=\emptyset$. If $g_2 \in \nQ_{i1} \cup \nQ_{i2}$ then by the definition of $\nQ_{i1}$ and $\nQ_{i2}$,  $\partial g_\infty \cap \partial g_2$ is clearly a $2$-ball. If $g_2 \in \nQ_{i3}$ then by the definition of $\nQ_{i3}$, 
there is $g_2' \in \nP_{i3}$ with $\level{g_2}=\level{g_2'}+1$ and $g_2\subset g_2'$ such that $\partial g_\infty \cap \partial g_2'$ is a union of $2$ disjoint $2$-balls. By Proposition~\ref{prop:capGraph}~(1) the intersection $\partial g_\infty \cap \partial g_2'$ is of type $\overline{R'}$ where $\overline{R'} \in \nI$. Lemma~\ref{lem:auxt1t2} now implies that 
there is an edge $\overline{R'}\rightarrow\overline{R}$ in $\nI$ such that $U(R')$ is a union of at least $2$ disjoint $2$-balls and $\partial g_\infty \cap \partial g_2 \simeq U(R)$. An inspection of the graph $\nI$ in Figures~\ref{Fig:B=2} and~\ref{Fig:B>2} shows that each successor of a node corresponding to $2$ disjoint $2$-balls corresponds to a single $2$-ball. Thus $\partial g_\infty \cap \partial g_2$ is a $2$-ball and item~(3) is proved.

To prove item (4) we note that by Lemma~\ref{lem:nBasic}~(i) each of the atoms $g_j \in \nQ_i(\mathbf{n})$ ($1\le j\le 3$) is a union of sets of the form $M^{-i-1}(T+z)^\circ$ with $z\in \Z^3$. This union is finite unless $g_j=g_\infty$. Thus $\partial g_1\cap\partial g_2 \cap \partial g_3$ is a finite (possibly empty) union of intersections of the form $M^{-i-1}((T+z_1) \cap (T+z_2) \cap (T+ z_3))$ with $z_1,z_2,z_3\in \Z^3$. By Proposition~\ref{lem:TZ:19}~(2) each of these intersections is either empty or homeomorphic to an arc. This proves item~(4).
\end{proof}

\subsection{An order on the subsets of an atom}\label{sec:order}
Let $T$ be an $ABC$-tile with $A=1$ having $14$ neighbors. Let $\mathbf{n}=(n_j)_{j\ge 1}$ be a sequence with $n_j\in \N\cup\{\infty\}$ satisfying $n_1\ge 3$ and  $n_{j+1}-n_j\ge 3$ and let $(\nQ_i(\mathbf{n}))_{i\ge 1}$ be the associated sequence of regular partitionings defined in \eqref{eq:nPn} (see Lemma~\ref{lem:nBasic}).
In this section we define an order on the sets $\{g' \in \nQ_{i+1}(\mathbf{n}) \,:\, g' \subseteq g\}$ of atoms in $\nQ_{i+1}(\mathbf{n})$ that are contained in some fixed $g \in \nQ_{i}(\mathbf{n})$ and prove some connectivity properties of related intersections ($i\ge 1$). 

Let $k\in \N$. If $z=(e_{k-1},\ldots,e_0)_M$ and $z'=(e'_{k-1},\ldots,e'_0)_M$ are elements of $\D_k$ we say that $z\prec z'$  if and only if $(e_{k-1},\ldots,e_0)<_{\mathrm{lex}}(e'_{k-1},\ldots,e'_0)$ in lexicographic order (so, for instance $(2,1,4)_M\prec(3,0,0)_M$ and $(0,2,3)_M\prec(0,2,4)_M$).
This defines an order on $\D_k$. By definition, this order has the following property. 
Let $k,k'\in \N$ with $k\le k'$ be given.
Let $M^{-k}(T+ d_1)$, $M^{-k}(T+ d_2)$ with $d_1,d_2\in \D_k$, $d_1\neq d_2$. If $M^{-k'}(T+ d'_\ell)$ with $d'_\ell \in \D_{k'}$ is a subtile of $M^{-k}(T+ d_\ell)$ for $\ell\in\{1,2\}$, then 
\begin{equation}\label{eq:d1d1primed2deprime}
d_1 \prec d_2 \quad (\hbox{in } \D_k)
\qquad\Longleftrightarrow\qquad 
d'_1 \prec d'_2 \quad(\hbox{in } \D_{k'}).
\end{equation}

We continue with two lemmas that will be needed in the proof of the connectivity result stated in Proposition~\ref{prop:orderP2}.

\begin{lem}\label{lem:3first}
Let $T$ be an $ABC$-tile and assume that  $A=1$.
Let $z=( e_2, e_1, e_0)_M\in \mathcal{D}_3$ be given. Then the following assertions hold (where ``$\prec$'' denotes the order on $\D_3$).
\begin{itemize}
\item $z+P \in \mathcal{D}_3$ with $z \prec z+P$ if and only if $ e_0 < C-1$.
\item $z+Q \in \mathcal{D}_3$ with $z \prec z+Q$ if and only if $ e_0<C-1$ and $ e_1<C-1$.
\item $z+N \in \mathcal{D}_3$ with $z \prec z+N$ if and only if $ e_0<C-B$, $ e_1<C-1$, and $ e_2<C-1$.
\item $z+Q-P \in \mathcal{D}_3$ with $z \prec z+Q-P$ if and only if $ e_1<C-1$.
\item $z+N-P \in \mathcal{D}_3$ with $z \prec z+N-P$ if and only if $ e_0<C-B+1$, $ e_1<C-1$, and $ e_2<C-1$.
\item $z+N-Q \in \mathcal{D}_3$ with $z \prec z+N-Q$ if and only if $ e_0<C-B+1$ and $ e_2<C-1$.
\item $z+N-Q+P \in \mathcal{D}_3$ with $z \prec z+N-Q+P$ if and only if $ e_0<C-B$ and $ e_2<C-1$.
\item Let $\alpha\in -\mathcal{S}_1$. Then $z+\alpha\in \mathcal{D}_3$ and $z\prec z+\alpha$ cannot hold simultaneously.
\end{itemize}
\end{lem} 

The proof is done easily by direct calculation; note that 
$(e_2,e_1,e_0)_M=\begin{pmatrix} e_0\\ e_1\\e_2 \end{pmatrix}$.

\begin{lem}\label{lem:3second}
Let $T$ be an $ABC$-tile with $14$ neighbors and assume that  $A=1$.
Let $j\in\{1,2,3\}$ and $z\in \mathcal{D}_j$ be given. Then
\[
U_{z,j}=(T + z) \cap \bigg(\partial (M^jT) \cup \bigcup_{\begin{subarray}{c} y \prec z \\ y\in \mathcal{D}_j \end{subarray}} (T+y)\bigg) 
\]
is a connected set (here ``$\prec$'' denotes the order on $\D_j$).
\end{lem}

\begin{proof}
The intersection in the statement of the lemma can be written as
\[
U_{z,j} = (T + z) \cap \bigg( \bigcup_{y\not\in \mathcal{D}_j}(T+y) \cup \bigcup_{\begin{subarray}{c} y \prec z \\ y\in \mathcal{D}_j \end{subarray}} (T+y)\bigg) 
=
\bigcup_{\alpha \in \mathcal{S} \setminus \mathcal{S}'} (T+z) \cap (T+z+\alpha),
\]
where
\[
\mathcal{S}' = \{\alpha\in \mathcal{S} \;:\; z+\alpha\in \mathcal{D}_j \hbox{ and } z\prec z+\alpha\}.
\]
We prove the case $j=3$. To this end let $z=( e_2, e_1, e_0)_M\in \mathcal{D}_3$. We have to distinguish 12 cases according to the inequalities occurring in Lemma~\ref{lem:3first}.
\begin{itemize}
\item[(i)] $ e_2\in \{0,\ldots, C-2\}$, $ e_1\in \{0,\ldots, C-2\}$, $ e_0\in \{0,\ldots,C-B-1\}$. According to Lemma~\ref{lem:3first} in this case we have $\mathcal{S}'= \mathcal{S}_1$ and, hence,
$
U_{z,3} = \bigcup_{\alpha \in -\mathcal{S}_1} (T+z) \cap (T+z+\alpha)
$
is homeomorphic to the (connected) $2$-ball \circled{$a^3$}.

\item[(ii)] $ e_2\in \{0,\ldots, C-2\}$, $ e_1\in \{0,\ldots, C-2\}$, $ e_0 = C-B$. Here Lemma~\ref{lem:3first} yields $\mathcal{S}'=\{P,Q,Q-P,N-P,N-Q\}$, hence, $\mathcal{S} \setminus \mathcal{S}'=-S_1 \cup\{N,N-Q+P\}$ and $U_{z,3}$ easily seen to be a $2$-ball by using Lemma~\ref{eq:UrO}.

\item[(iii)] $ e_2\in \{0,\ldots, C-2\}$, $ e_1\in \{0,\ldots, C-2\}$, $ e_0 = \{C-B+1,\ldots, C-2\}$. Lemma~\ref{lem:3first} yields $\mathcal{S}'=\{P,Q,Q-P\}$ and $U_{z,3}$ is homeomorphic to the $2$-ball \circled{$a^2$}.

\item[(iv)] $ e_2\in \{0,\ldots, C-2\}$, $ e_1\in \{0,\ldots, C-2\}$, $ e_0 = C-1$. Here $\mathcal{S}'=\{Q-P\}$ and $U_{z,3}$ is homeomorphic to the $2$-ball $\partial T\setminus$\circled{$g^4$}.

\item[(v)] $ e_2\in \{0,\ldots, C-2\}$, $ e_1=C-1$, $ e_0\in \{0,\ldots,C-B-1\}$. Here $\mathcal{S}'=\{P,N-Q,N-Q+P\}$ and $U_{z,3}$ is homeomorphic to the $2$-ball $\partial T\setminus$\circled{$d^4$}.

\item[(vi)] $ e_2\in \{0,\ldots, C-2\}$, $ e_1=C-1$, $ e_0=C-B$. Here $\mathcal{S}'=\{P,N-Q\}$ and $U_{z,3}$ is homeomorphic to $\mathbb{S}^1\times[0,1]$ by Lemma~\ref{eq:UrO} and, hence, it is connected.

\item[(vii)] $ e_2\in \{0,\ldots, C-2\}$, $ e_1=C-1$, $ e_0=\{C-B+1,\ldots, C-2\}$. Here $\mathcal{S}'=\{P\}$ and $U_{z,3}$ is homeomorphic to the $2$-ball \circled{$a^1$}.

\item[(viii)] $ e_2\in \{0,\ldots, C-2\}$, $ e_1=C-1$, $ e_0=C-1$. Here $\mathcal{S}'=\emptyset$ and $U_{z,3}$ is homeomorphic to the $2$-sphere by Lemma~\ref{eq:UrO} and, hence, it is connected.

\item[(ix)] $ e_2=C-1$, $ e_1\in \{0,\ldots, C-2\}$, $ e_0\in \{0,\ldots,C-2\}$. Here $\mathcal{S}'= \{P,Q,Q-P\}$ and, hence, $U_{z,3}$ is  homeomorphic to the $2$-ball \circled{$a^2$}.

%
%

\item[(x)] $ e_2=C-1$, $ e_1\in \{0,\ldots, C-2\}$, $ e_0 = C-1$. Here $\mathcal{S}'=\{Q-P\}$ and $U_{z,3}$ is homeomorphic to the $2$-ball $\partial T\setminus$\circled{$g^4$}.

\item[(xi)] $ e_2=C-1$, $ e_1=C-1$, $ e_0\in \{0,\ldots,C-2\}$. Here $\mathcal{S}'=\{P\}$ and $U_{z,3}$ is homeomorphic to the $2$-ball \circled{$a^1$}.

%
%

\item[(xii)] $ e_2=C-1$, $ e_1=C-1$, $ e_0=C-1$. Here $\mathcal{S}'=\emptyset$ and $U_{z,3}$ is homeomorphic to the $2$-sphere.
\end{itemize}

The proof for the cases $j\in\{1,2\}$ is similar but easier than the case $j=3$ and we omit it.
\end{proof}

We are now ready to prove the following proposition.  Note that the property proved in this result is related to the condition stated in Proposition~\ref{prop:bing3}~(4).

\begin{prop}\label{prop:orderP2}
Let $T$ be an $ABC$-tile with $14$ neighbors and assume that  $A=1$. Let $\mathbf{n}=(n_j)_{j\ge 1}$ be a sequence with $n_j\in \N\cup\{\infty\}$ satisfying $n_1\ge 3$ and  $n_{j+1}-n_j\ge 3$.
Let $i\ge 1$ and let $g\in\nQ_i(\mathbf{n})$ be given. The set $\{g_1,\ldots,g_n\}\subseteq \nQ_{i+1}(\mathbf{n})$ of all atoms of $\nQ_{i+1}(\mathbf{n})$ that are subsets of $g$ can be ordered in a way that $\partial g_j \cap (\partial g \cup \partial g_1\cup\dots\cup \partial g_{j-1})$ is connected for each $j\in \{1,\ldots, n\}$.
\end{prop}

\begin{proof}
If $n=1$ (which is true in particular for $g=g_\infty$) the result is trivial. If $n>1$ then $g=M^{-k}(T^\circ+z)$ for some $k\ge 0$ and some $z\in \D_k$. For convenience, we set $g'=M^kg-z=T^\circ$ and $g_j'=M^kg_j-z$ for $j\in\{1,\ldots, n\}$. By Lemma~\ref{lem:nBasic}~(i) we know that $g'_j = M^{-k_j}(T^\circ+y_j)$ with $k_j\in\{1,2,3\}$ and $y_j \in \D_{k_j}$. We assume that $\{g'_1,\ldots,g'_n\}$ is ordered in a way that the following is true: For each $j$ subdivide $\overline{g_j'}$ in subtiles of the form $M^{-3}(T+d)$ with $d\in \D_3$ by the set equation \eqref{eq:iset}. Let $j_1,j_2\in \{1,\ldots,n\}$ be distinct and let $M^{-3}(T+d_\ell)$ be a subtile of $\overline{g_{j_\ell}'}$ ($\ell\in\{1,2\}$). Then $d_1 \prec d_2$ w.r.t.\ the order in $\D_3$ if and only if $j_1<j_2$. 

Note that
\begin{equation}\label{eq:primeoverline}
\partial g_j \cap (\partial g \cup \partial g_1\cup\dots\cup \partial g_{j-1}) \simeq
\partial g'_j \cap (\partial g' \cup \partial g'_1\cup\dots\cup \partial g'_{j-1})=
\overline{g'_j} \cap (\partial{g'} \cup \overline{g'_1}\cup\dots\cup \overline{g'_{j-1}})
\end{equation}
holds for each $j\in \{1,\ldots, n\}$ (the equality holds because the sets $\overline{g'_1},\ldots,\overline{g'_n}$ cover $\overline{g'}$ overlapping only at their boundaries). Moreover, we have
\begin{align}
\overline{g'_j} \cap  (\partial g' \cup \overline{g'_1}\cup\dots\cup \overline{g'_{j-1}})
&=  M^{-k_j}(T + y_j) \cap \bigg(\partial T \cup \bigcup_{\ell=1}^{j-1} M^{-k_\ell}(T+y_\ell) \bigg)
\nonumber
\\
&=
M^{-k_j}(T + y_j) \cap \bigg(\partial T \cup \bigcup_{\begin{subarray}{c} y \prec y_j \\ y\in \mathcal{D}_{k_j} \end{subarray}} M^{-k_j}(T+y)\bigg) 
\label{eq:3stepconn}
\\
&=
M^{-k_j} \bigg((T + y_j) \cap \bigg(\partial (M^{k_j}T) \cup \bigcup_{\begin{subarray}{c} y \prec y_j \\ y\in \mathcal{D}_{k_j} \end{subarray}} (T+y)\bigg) \bigg).
\nonumber
\end{align}
In the second equality we used the set equation \eqref{eq:iset} to subdivide (if $k_\ell < k_j$) or group (if $k_\ell > k_j$) the sets $M^{-k_\ell}(T+y_\ell)$ into sets that are all of the form $M^{-k_j}(T + y)$ for some $y\in \D_{k_j}$.   By the ordering of $\{g'_{1},\ldots, g'_{n}\}$ and by property \eqref{eq:d1d1primed2deprime} of ``$\prec$'' this yields the union over all sets of the form $M^{-k_j}(T+y)$ with $ y \prec y_j$, $y\in \mathcal{D}_{k_j} $. 

Because $k_j\in\{1,2,3\}$, the last set in \eqref{eq:3stepconn} is connected by Lemma~\ref{lem:3second} and the result follows from \eqref{eq:primeoverline} and \eqref{eq:3stepconn}.
\end{proof}

\subsection{Proof of Theorem~\ref{thm:ball}}\label{sec:proof1}
We have to show that $T$ is homeomorphic to $\mathbb{D}^3$ under the conditions of Theorem~\ref{thm:ball}. In view of Section~\ref{sec:normal} we may assume that $T$ is an $ABC$-tile with $14$ neighbors and that  $A=1$. Throughout the proof we will use the fact that $T$ is a Peano continuum by Lemma~\ref{lem:HataPeano}.

Our strategy is to construct a decreasing sequence of partitionings of $\mathbb{D}^3$ that is equivalent to $(\nQ_i'(\mathbf{n}))_{i\ge 1}$ for a suitable sequence $\mathbf{n}=(n_j)_{j\ge 1}$ with $n_j\in \N$ (finite) satisfying $n_1\ge 3$ and  $n_{j+1}-n_j\ge 3$. Then the result will follow from Lemma~\ref{lem:homeo}. We will use the theory of partitionings due to Bing~\cite{Bing51}. Bing gives a topological characterization of $3$-spheres in terms of decreasing sequences of regular partitionings. In Theorem~\ref{thm:ball} we deal with $3$-balls instead of $3$-spheres. However, the main difference between Bing's setting and our's is that, contrary to his assumptions (see \cite[Theorem~1~(1.2)]{Bing51} and  the discussion in ~\cite[p.~25]{Bing51}), we do not have that for $g_1,g_2\in \bigcup_{i\ge 1} \nQ_i'(\mathbf{n})$ the intersection $\partial g_1 \cap \partial g_2$ is either empty or homeomorphic to $\mathbb{D}^2$. We have to settle for the weaker results in Proposition~\ref{lem:topologyP} (2) and (3). To make up for this we will exploit the self-affinity of~$T$. This difference is the reason why we cannot use Bing's original proof here.

The following lemma contains the crucial tool for the proof of Theorem~\ref{thm:ball}.

\begin{lem}\label{lem:HKnew}
There is a sequence $\mathbf{n}=(n_j)_{j\ge 1}$ with $n_j\in\N$, $n_j \ge 3$ and $n_{j+1}-n_j \ge 3$ such that there are sequences $(H_i)_{i\ge 1}$ and $(K_{n_j})_{j\ge 1}$ of partitionings of $\mathbb{R}^3$ with the following properties.
\begin{itemize}
\item[(i)] For each $h\in H_i$ the boundary $\partial h$ is a tame $2$-sphere in $\mathbb{R}^3$ ($i\ge 1$).
\item[(ii)] $H_{i+1}$ is a refinement of $H_i$ for each $i\ge 1$.
\item[(iii)] $h_0=\mathbb{R}^3\setminus \mathbb{D}^3$ is an atom of $H_i$ for each $i\ge 1$.
\item[(iv)] $(H_i\setminus \{h_0\})_{i\geq 1}$ is equivalent to $(\nQ'_i (\mathbf{n}))_{i \geq 1}$ in the sense of Definition~\ref{Def:Homeo}. 
\item[(v)] There is a sequence $(F_i)_{i\ge 1}$ where  $F_{i} : \bigcup_{g\in \nQ_{i}'(\mathbf{n})} \partial g \to \bigcup_{h\in H_{i}\setminus\{h_0\}} \partial h$ is a homeomorphism with the following properties: If $i > 1$ then the restriction of $F_{i}$ to $\bigcup_{g\in \nQ_{i-1}'(\mathbf{n})} \partial g$ is equal to $F_{i-1}$. If $i\ge 1$ then for each $g\in \nQ_{i}'(\mathbf{n})$ we have $F_{i}(\partial g) = \partial h$, where $h\in H_{i}\setminus\{h_0\}$ is the atom corresponding to $g$ under the equivalence in~(iv). 
\item[(vi)] For each $k\in K_{n_j}$ the boundary $\partial k$ is a tame $2$-sphere in $\mathbb{R}^3$ ($j\ge 1$).
\item[(vii)] $K_{n_{j+1}}$ is a refinement of $K_{n_j}$ for each $j\ge 1$.
\item[(viii)]$h_0$ is an atom of $K_{n_j}$ for each $j\ge 1$.
\item[(ix)] $(K_{n_j})_{j\geq 1}$ is equivalent to $(H_{n_j})_{j \geq 1}$ in the sense of Definition~\ref{Def:Homeo}. 
\item[(x)] $(K_{n_j} \setminus\{h_0\})_{j\ge 1}$ is a decreasing sequence of partitionings of $\mathbb{D}^3$. 
\end{itemize}
\end{lem}

\begin{proof}
The proof splits in two parts. The first part is an induction proof in which we construct the sequences $(H_i)_{i\ge 1}$ and $(K_{n_j})_{j\ge 1}$, where the second sequence might {\it a priori} be finite. In the second part of the proof we show that $(K_{n_j})_{j\ge 1}$ is in fact an infinite sequence.

\medskip

We say that $\mathcal{A}(m)$ holds if there exist $j_0=j_0(m) \in \N$,  $n_1,\ldots, n_{j_0} \le m$, $n_{j}=\infty$ for $j>j_0$, and finite sequences $(H_i)_{i= 1}^m$, $(K_{n_j})_{j=1}^{j_0}$ of partitionings of $\mathbb{R}^3$ such that the following properties hold. (We set $n_0=0$  and $K_0=\{h_0,\R^3\setminus\{h_0\}\}$ for convenience.) 
\begin{itemize}
\item[(i-$m$)] For each $h\in H_i$ the boundary $\partial h$ is a tame $2$-sphere in $\mathbb{R}^3$ $(1\le i \le m)$.
\item[(ii-$m$)] $H_{i+1}$ is a refinement of $H_i$ ($1\le i < m$).
\item[(iii-$m$)] $h_0=\mathbb{R}^3\setminus \mathbb{D}^3$ is an atom of $H_i$ ($1\le i \le m$).
\item[(iv-$m$)] $(H_i\setminus \{h_0\})_{i= 1}^m$ is equivalent to $(\nQ_{i}'(\mathbf{n}))_{i =1}^m$ in the sense of Definition~\ref{Def:Homeo}. 
\item[(v-$m$)] There is a sequence $(F_i)_{i= 1}^{m}$ where  $F_{i} : \bigcup_{g\in \nQ_{i}'(\mathbf{n})} \partial g \to \bigcup_{h\in H_{i}\setminus\{h_0\}} \partial h$ is a homeomrophism with the following properties:  If $1<i\le m$ then the restriction of $F_{i}$ to $\bigcup_{g\in \nQ_{i-1}'(\mathbf{n})} \partial g$ is equal to $F_{i-1}$. If $1\le i \le m$ then for each $g\in \nQ_{i}'(\mathbf{n})$ we have $F_{i}(\partial g) = \partial h$, where $h\in H_{i}\setminus\{h_0\}$ is the atom corresponding to $g$ under the equivalence in (iv-$m$). 
\item[(vi-$m$)] For each $k\in K_{n_j}$ the boundary $\partial k$ is a tame $2$-sphere in $\mathbb{R}^3$ ($1\le j\le j_0$).
\item[(vii-$m$)] $K_{n_{j+1}}$ is a refinement of $K_{n_j}$ $(1 \le j < j_0)$.
\item[(viii-$m$)]$h_0$ is an atom of $K_{n_j}$ $(1\le j\le j_0)$.
\item[(ix-$m$)] $(K_{n_j})_{j=1}^{j_0}$ is equivalent to $(H_{n_j})_{j=1}^{j_0}$ in the sense of Definition~\ref{Def:Homeo}. 
\item[(x-$m$)] 
There exist homeomorphisms $f_1,\ldots, f_{j_0}: \R^3\to\R^3$  such that each boundary point of each atom of $K_{n_{j-1}}$ is invariant under $f_j$ and $f_j\circ \cdots \circ f_1$ keeps $\R^3\setminus \mathbb{D}^3$ invariant and carries each other atom of $H_{n_j}$ into a set of diameter less than $\frac 1{2^j}$. Moreover, 
\[
K_{n_j}=\{f_j\circ \cdots \circ f_2\circ f_1(h)\;:\; h\in H_{n_j}\}
\]
($1\le j\le j_0$).
Thus $K_{n_j}\setminus\{h_0\}$ is a partitioning of $\mathbb{D}^3$ with $\max\{ {\rm diam}(k) \,:\, k\in K_{n_j}\}<\frac1{2^j}$.
\end{itemize}
To prove the lemma we first show by induction that $\mathcal{A}(m)$ is true for all $m\ge 1$. 
In the course of this induction proof we construct sequences $(H_i)_{i\ge 1}$ and $(K_{n_j})_{j\ge 1}$ that satisfy (i-$m$) -- (x-$m$) for each $m\in \N$. This induction argument implies (i) -- (v). To gain (vi) -- (x) we have to show that our construction leads to $j_0(m)\nearrow \infty$ for $m\to\infty$.  

\medskip

For the induction start we prove $\mathcal{A}(1)$. 
Set $j_0(1)=0$. Thus $n_j=\infty$ for all $j\ge 1$ and $(K_{n_j})_{j=1}^0$ is the empty sequence.  
Set $H_1=\{(\mathbb{D}^3)^\circ,\mathbb{R}^3\setminus\mathbb{D}^3\}$. Then $\bigcup_{g\in \nQ_{1}'(\mathbf{n})} \partial g= \bigcup_{g\in \nQ_{1}'} \partial g=\partial T$ and  $\bigcup_{h\in H_{1}\setminus\{h_0\}} \partial h=\partial\mathbb{D}^3$. Since $\partial T$ is a $2$-sphere by \cite[Theorem~1.1]{TZ:19}, there exists a homeomorphism $F_1:\partial T \to \partial \mathbb{D}^3$. Thus $H_1$ satisfies (i-$1$), (ii-$1$) (which is empty for $m=1$), (iii-$1$), (iv-$1$) (note that $\nQ'_1(\mathbf{n})=\{T^\circ\}$; thus $T^\circ$ corresponds to $(\mathbb{D}^3)^\circ$), and (v-$1$) (whose first assertion is empty for $m=1$). Since $j_0(1)=0$, assertions (vi-$1$) -- (x-$1$) are empty. This concludes the induction start.

\medskip

To perform the induction step, let $m\ge 1$ and assume that $\mathcal{A}(m)$ is true. We have to distinguish two cases.

\medskip

\noindent {\it Case 1:} For $j_0=j_0(m)$ we have $m\ge n_{j_0}+2$ and there exists a homeomorphism $f_{j_0+1}:\R^3 \to \R^3$ such that each boundary point of each atom of $K_{n_{j_0}}$ is invariant under $f_{j_0+1}$ and $f_{j_0+1}\circ \cdots \circ f_1$ keeps $\R^3\setminus \mathbb{D}^3$ invariant and carries each other atom of $H_{m}$ into a set of diameter less than $\frac 1{2^{j_0+1}}$.

\medskip

\noindent {\it Case 2:} $m< n_{j_0(m)}+2$ or a homeomorphism as in Case~1 does not exist (this is the complement of Case~1). 

\medskip

If {\em Case 1} is in force then set $j_0(m+1)=j_0(m)+1$ and $n_{j_0(m+1)}=m+1$. This has no effect on the partitionings $\nQ_{1}'(\mathbf{n}),\ldots,\nQ_{m}'(\mathbf{n})$. By the definition of $(\nQ_{i}'(\mathbf{n}))_{i\ge 1}$ in \eqref{eq:nPn} we have 
\[
\nQ_{m}'(\mathbf{n})=\nQ_{n_{j_0(m+1)-1}}'(\mathbf{n})=\nQ_{n_{j_0(m+1)}}'(\mathbf{n})=\nQ_{m+1}'(\mathbf{n}).
\]
Thus, setting $H_{m+1}=H_{m}$ trivially yields (i-$(m+1)$) -- (v-$(m+1)$) from (i-$m$) -- (v-$m$).
Now let $f_{j_0(m)+1}=f_{j_0(m+1)}$ be the homeomorphism having the properties specified in Case~1 and set $K_{m+1}=\{f_{j_0(m+1)}\circ \cdots \circ f_2\circ f_1(h)\,:\, h\in H_{m+1}\}$. The condition for Case~1 (here we use that $H_{m+1}=H_m$) together with (x-$m$) implies that (x-$(m+1)$) is true. Because $f_{j_0(m+1)}\circ \cdots \circ f_2\circ f_1$ is a homeomorphism that keeps $\R^3\setminus\mathbb{D}^3$ invariant, (vi-$(m+1)$), (viii-$(m+1)$), and (ix-$(m+1)$) follow. Finally, (vii-$(m+1)$) is true because $f_{j_0(m+1)}$ leaves each boundary point of $K_{n_{j_0(m)}}$ invariant by (x-$(m+1)$). This finishes the induction step for Case~1.

\medskip

If {\em Case 2} holds, set $j_0(m+1)=j_0(m)$. Let $a\in  \nQ_{m}'(\mathbf{n})$ and let $\{g_{a,1},g_{a,2},\dots,g_{a,n(a)}\}=\{g \in \nQ'_{m+1}(\mathbf{n})\,:\, g \subseteq a\}$ be ordered in a way that they satisfy the conclusion of Proposition~\ref{prop:orderP2}. Let $h(a)\in H_{m}$ be the element corresponding to $a$ via (iv-$m$).
We want to apply Proposition~\ref{prop:bing3} with $S=\mathbb{S}^3$,  $S_2=\partial_{\mathbb{S}^3} a$, $C=\partial_{\mathbb{R}^3} h(a)$, $G=\nQ_{m+1}(\mathbf{n})$, and $F=F_{m}|_{\partial a}$. Therefore, we have to check the conditions of this proposition. By Lemma~\ref{lem:nBasic}~(ii), $\nQ_{m+1}(\mathbf{n})$ is a regular partitioning of $\mathbb{S}^3$ and Proposition~\ref{lem:topologyP} implies that $\nQ_{m+1}(\mathbf{n})$ satisfies conditions (1), (2), and (3) of Proposition~\ref{prop:bing3} (note that $m+1\ge 2$). By the order we chose for the elements $g_{a,1},\ldots, g_{a,n(a)}$ (using Proposition~\ref{prop:orderP2}), the set $\{g_{a,1},\dots,g_{a,n(a)}\}$ satisfies condition (4) of Proposition~\ref{prop:bing3} (observe that $a=\big(\overline{\bigcup_{\ell=1}^{n(a)} g_{a,\ell}}\big)^{\circ}$). Thus we can apply Proposition~\ref{prop:bing3}. This yields a partitioning $H_{m+1,a}=\{h_{a,0},h_{a,1},\dots,h_{a,n(a)}\}$ of $\R^3$, where $h_{a,0}$ is the exterior of $\partial h(a)$ and the boundaries of $h_{a,j}$ are tame $2$-spheres, and a homeomorphism 
 \begin{equation*}
 F_{m+1,a}: \partial_{\mathbb{S}^3} (g_{a,1}\cup \dots \cup g_{a,n(a)})\to \partial_{\mathbb{R}^3}(h_{a,1}\cup \dots\cup h_{a,n(a)})
 \end{equation*}
satisfying 
\begin{equation}\label{eq:Frest}
F_{m+1,a}|_{\partial a} =F_{m}|_{\partial a} 
\end{equation}
and 
\begin{equation}\label{eq:Fbound}
F_{m+1,a} (\partial {g_{a,j}})=\partial h_{a,j} \quad\hbox{for each } j\in \{1,2,\dots,n(a)\}.
\end{equation}

Set $H_{m+1}=\{h_0\}\cup\bigcup_{a\in  \nQ_{m}'(\mathbf{n})} H_{m+1,a}'$, where $H_{m+1,a}'=H_{m+1,a}\setminus\{h_{a,0}\}$. By construction, $H_{m+1}$ is a partitioning of $\R^3$ whose atoms have tame spherical  boundary, which is a refinement of $H_{m}$, and which contains the atom $h_0$. Thus, by the induction hypothesis $\mathcal{A}(m)$, $(H_i)_{i=1}^{m+1}$ satisfies (i-$(m+1)$), (ii-$(m+1)$), and (iii-$(m+1)$). Observe that
\begin{align*}
F_{m+1}:
 \bigcup_{g\in \nQ_{m+1}'(\mathbf{n})} \partial g &\to  \bigcup_{h\in H_{m+1}\setminus \{h_0\}} \partial h, \qquad
x 
\mapsto F_{m+1,a}(x) \quad\hbox{for}\quad x\in \partial_{\mathbb{S}^3} (g_{a,1}\cup \dots \cup g_{a,n(a)})
\end{align*}
is a homeomorphism which is well-defined on the boundary of each atom of the partitioning $\nQ_{m+1}'(\mathbf{n})$ because the homeomorphisms $F_{m+1,a}$, $a\in\nQ'_m(\mathbf{n})$, agree on the intersections of their domains. Thus (iv-$(m+1)$) holds with the correspondence $g_{a,\ell}\leftrightarrow h_{a,\ell}$ ($a\in \nQ_m(\mathbf{n})$, $1\le\ell\le n(a)$; see in particular \eqref{eq:Fbound}).
To see that (v-$(m+1)$) is true, note that the restriction of $F_{m+1}$ to the domain of $F_{m}$ equals $F_{m}$ by \eqref{eq:Frest} and boundaries of corresponding atoms are mapped bijectively to each other by \eqref{eq:Fbound}. Thus (i-$(m+1)$) -- (v-$(m+1)$) hold also in Case~2.

In Case 2~we have $j_0(m+1)=j_0(m)$. Thus items (vi-$(m+1)$) -- (x-$(m+1)$) are the same as (vi-$m$) -- (x-$m$) and there is nothing to prove. Thus the induction step is finished also in Case~2. 
This completes the induction proof. 

\medskip 

This induction proof already implies assertions (i) -- (v) of the lemma. To get (vi) -- (x) it remains to show that our process defines an infinite sequence $(n_j)_{j\ge 1}$ of integers $n_j$, {i.e.}, that $j_0(m)\nearrow \infty$ for $m\to\infty$. The monotonicity of $j_0(m)$ is clear from the construction. Since $j_0(m+1)=j_0(m)+1$ whenever we are in Case~1, it remains to prove the following claim.

\medskip

\noindent{\it Claim:} Case 1 occurs for infinitely many $m$ in the above induction process.

\medskip

To prove this assume on the contrary that Case~1 occurs only finitely many times. Then either there is a largest $m$ that has $m=n_{j_0}$ for some $j_0\ge1$, or Case~1 never occurs; then we set $m=1$ and $j_0=0$. Let $g \in \nQ'_m(\mathbf{n})$ and let $h$ be the element of $H_m\setminus\{h_0\}$ corresponding to $g$. 
Let
\[
(K_i(h))_{i>m} = (\{f_{j_0}\circ \cdots\circ f_1(h') \;:\; h' \in  H_i \hbox{ with } h' \subseteq  h\}\cup\{\R^3\setminus\overline{f_{j_0}\circ \cdots\circ f_1(h)}\})_{i>m}. 
\]
By the definition of $(\nQ_i(\mathbf{n}))_{i\ge 1}$ and by (iv) we have that
\begin{equation}\label{eq:hpequ}
\begin{split}
(K_i(h))_{i>m} &\sim (\{h'\in  H_i \;:\; h' \subseteq  h\} \cup\{\R^3\setminus\overline{h}\} )_{i>m} 
\\&
\sim (\{g'\in   \nQ_i'(\mathbf{n}) \;:\; g' \subseteq  g\}\cup\{\R^3\setminus\overline{g}\})_{i>m} 
\\&
= ([g](\nQ'_{i-\level{g}})\cup\{\R^3\setminus\overline{g}\})_{i>m} 
\\&
\sim 
 \begin{cases}
 (\nQ_i)_{i\ge 1} & \hbox{if } \level{g}=m, \\
 (\nQ_i)_{i\ge 2} & \hbox{if } \level{g}=m-1,
 \end{cases} 
 \end{split}
 \end{equation}
 where the equivalences have the additional property that $\partial a_1 \cap \partial a_2 \simeq \partial a_1'\cap \partial a_2'$ if $a_\ell$ and $a_\ell'$ are corresponding elements ($1\le \ell \le 2$). Indeed, this homeomorpies hold by (v) and because $f_{j_0}\circ\cdots\circ f_{1}$ and $[g]$ are homeomorphisms from $\R^3$ to $\R^3$.
Note that $(K_i(h))_{i>m}$ satisfies the conditions of Proposition~\ref{prop:bing5} with $C=\partial f_{j_0}\circ \cdots\circ f_1(h)$:  Proposition~\ref{prop:bing5}~(1) holds by (i), Proposition~\ref{prop:bing5}~(2)\footnote{It is this condition (2) of Proposition~\ref{prop:bing5} that required us to work with $(\nQ_i'(\mathbf{n}))$ rather than $(\nQ_i')$. Indeed, the fact that we use $(\nQ_i'(\mathbf{n}))$ guarantees that \eqref{eq:hpequ} holds.}
 holds by \eqref{eq:hpequ} and Proposition~\ref{lem:topologyP}~(3) (for $i\ge 2$; for $\nQ_1'$ it is easy to see), Proposition~\ref{prop:bing5}~(3) is true by (ii), Proposition~\ref{prop:bing5}~(4) is obviously true, and Proposition~\ref{prop:bing5}~(5) holds by (v). Indeed, note that $(\nQ_i'(\mathbf{n}))$ is decreasing and $F_k$ preserves $F_{k+n}$ on the boundaries of the elements of $\nQ_k'(\mathbf{n})$ for each $n\in\N$. Applying Proposition~\ref{prop:bing5} to $(K_i(h))_{i>m}$ we see that there is an integer $m'(g) \ge m+2$ for which there is a homeomorphism $f_{g,j_0+1}:\R^3\to \R^3$ that leaves $\R^3\setminus f_{j_0}\circ \cdots\circ f_1(h)$ pointwise invariant and $\mathrm{diam}(f_{g,j_0+1}(k'))<\frac 1{2^{j_0+1}}$ holds for each $m''\ge m'(g)$ and each $k' \in K_{m''}(h) \setminus \{\R^3\setminus\overline{f_{j_0}\circ \cdots\circ f_1(h)}\}$. Doing this for each $g\in \nQ'_m(\mathbf{n})$ and choosing $m'=\max\{m'(g)\,:\, g \in\nQ'_m(\mathbf{n})\}+3$ we can define the homeomorphism $f_{j_0+1}:\R^3\to \R^3$ by
\[
f_{j_0+1}(x) = f_{g,j_0+1}(x) \quad \hbox{for $x \in f_{j_0}\circ \cdots\circ f_1(h)$, where $h\in H_m\setminus\{h_0\}$ corresponds to $g\in \nQ'_m(\mathbf{n})$} 
\]
(extending it continuously to $\mathbb{R}^3$ by the identity outside $\mathbb{D}^3$). By construction, each boundary point of each atom of $K_{n_{j_0}}$ is invariant under $f_{j_0+1}$ and $f_{j_0+1}\circ \cdots \circ f_1$ keeps $\R^3\setminus \mathbb{D}^3$ invariant and carries each other atom of $H_{m'}$ into a set of diameter less than $\frac 1{2^{j_0+1}}$. Because $m' \ge  m+3$ we are in Case~1 for $m'>m$, a contradiction to the maximality of $m$. This proves the claim and, hence, the lemma.
\end{proof}

We can now easily finish the proof of Theorem~\ref{thm:ball}.
By Lemma~\ref{lem:nBasic} and Lemma~\ref{lem:HKnew}~(x), there is a strictly increasing sequence $(n_j)$ of positive integers such that $(\nQ_{n_j}'(\mathbf{n}))_{j \geq 1}$ and $(K_{n_j}\setminus\{h_0\})_{j \geq 1}$  are decreasing sequences of partitionings of $T$ and $\mathbb{D}^3$, respectively. From Lemma~\ref{lem:HKnew}~(iv) and~(ix) we obtain that $(\nQ_{n_j}'(\mathbf{n}))_{j\ge 1}\sim(H_{n_j}\setminus\{h_0\})_{j\ge 1}\sim (K_{n_j}\setminus\{h_0\})_{j\ge 1}$. Thus Lemma \ref{lem:homeo}  (see also Remark~\ref{rem:simeq}) implies that $T$ is homeomorphic to~$\mathbb{D}^3$. This concludes the proof of Theorem~\ref{thm:ball}.

\subsection{Proof of Theorem~\ref{thm:CW}} \label{sec:proof2}
In view of Section~\ref{sec:normal} we may assume that $T$ is an $ABC$-tile with $14$ neighbors and that  $A=1$. As in the proof of Lemma~\ref{eq:UrO} we see that the truncated octahedron $O$ is a CW complex in the following natural sense. Let $O_{\boldsymbol{\alpha}}$ (${\boldsymbol{\alpha}}\subseteq \nS$) be as in \ref{eq:Oa}. For $i\in\{0,1,2,3\}$ the closed $i$-cells are given by the nonempty sets $O_{\boldsymbol{\alpha}}$ with $\boldsymbol{\alpha}\subseteq\nS$ and $\#\boldsymbol{\alpha} = 3-i$. Thus the $0$-skeleton $O^0$ is the set of vertices of $O$. Each closed $1$-cell $O_{\{\alpha_1,\alpha_2\}}$ is attached to the two closed $0$-cells $O_{\boldsymbol{\alpha}}$ satisfying $\boldsymbol{\alpha} \supset \{\alpha_1,\alpha_2\}$ and $\#\boldsymbol{\alpha} = 3$ (these $2$ closed $0$-cells form a $0$-sphere, {\it i.e.}, two points). This yields the $1$-skeleton $O^1$ ({\it i.e.}, the edges of $O$). To get the $2$-skeleton $O^2$ (whose support is $\partial O$) we attach each closed $2$-cell $O_{\alpha_1}$, $\alpha_1\in\nS$, to the $1$-sphere $\bigcup_{\alpha_2 \in \nS :\alpha_2\not=\alpha_1} O_{\{\alpha_1,\alpha_2\}}$. Finally, we attach the closed $3$-cell $O=O_\emptyset$ to the sphere $O^2$.

From Proposition~\ref{lem:TZ:19} and Theorem~\ref{thm:ball} we see that the set $T$ is a CW complex whose closed $i$-cells are given by the nonempty sets $\B_{\boldsymbol{\alpha}}$ with $\boldsymbol{\alpha}\subseteq\nS$ and $\#\boldsymbol{\alpha} = 3-i$ for $i\in\{0,1,2,3\}$ with analogous attaching rules as above. 

Thus, by Lemma~\ref{lem:OB} and Theorem~\ref{thm:ball}, $T$ has the CW complex structure indicated in the statement of Theorem~\ref{thm:CW}. This CW complex structure is isomorphic to the natural CW complex structure of $O$. 
The number of closed $i$-cells asserted in Theorem~\ref{thm:CW} can immediately be counted on $O$: A truncated octahedron has $14$~faces, $36$~edges, and $24$~vertices. 
This finishes the proof of Theorem~\ref{thm:CW}.

\bibliographystyle{siam}  
\bibliography{biblio}

\end{document}